\documentclass[11pt]{amsart}

\usepackage[T1]{fontenc}
\usepackage[english]{babel}
\usepackage[utf8]{inputenc}

\usepackage{amsmath, amsthm, amssymb, amsfonts}
\usepackage{thmtools, thm-restate}

\usepackage{enumerate}
\usepackage{url}                
\usepackage{framed,color}
\definecolor{shadecolor}{RGB}{255,255,255}
\usepackage[a4paper,margin=3cm]{geometry}

\usepackage[all]{xy}
\usepackage{tikz}
\usetikzlibrary{arrows,chains,matrix,positioning,scopes,cd}

\providecommand*{\A}{\mathbb{A}}
\providecommand*{\C}{\mathbb{C}}

\providecommand*{\N}{\mathbb{N}}

\providecommand*{\Z}{\mathbb{Z}}

\providecommand*{\dd}{\mathbf{d}}
\providecommand*{\ee}{\mathbf{e}}

\providecommand*{\cA}{\mathcal{A}}
\providecommand*{\cB}{\mathcal{B}}
\providecommand*{\cC}{\mathcal{C}}

\providecommand*{\cF}{\mathcal{F}}

\providecommand*{\cN}{\mathcal{N}}
\providecommand*{\cO}{\mathcal{O}}

\providecommand*{\cR}{\mathcal{R}}

\providecommand*{\add}{\mathrm{add}}
\providecommand*{\gen}{\mathrm{gen}}
\providecommand*{\cogen} {\mathrm{cogen}}
\providecommand*{\proj}{\mathrm{proj}}

\providecommand*{\fdmod}{\mathrm{mod}}
\providecommand*{\modu}{\mathrm{mod}}

\providecommand*{\mon}{\mathrm{mon}}

\providecommand*{\Hom}{\mathrm{Hom}}

\providecommand*{\ext}{\mathrm{Ext}}

\providecommand*{\rad}{\mathrm{rad}}

\providecommand*{\GL}{\mathrm{Gl}}
\providecommand*{\aut}{\mathrm{Aut}}

\providecommand*{\NsQ}{\mathrm{N}_s(Q)}
\providecommand*{\NtwoQ}{\mathrm{N}_2(Q)}

\providecommand*{\soc}{\mathrm{soc}}
\providecommand*{\Top}{\mathrm{top}}
\providecommand*{\End}{\mathrm{End}}

\providecommand*{\add}{\mathrm{add}}
\providecommand*{\Aut}{\mathrm{Aut}}
\providecommand*{\Dim}{\underline{\mathrm{dim }}\;}
\providecommand*{\kdual}{\mathrm{D}}

\providecommand*{\rep}{\mathrm{R}}
\providecommand*{\repdd}{\mathrm{R}_d^{\dd}}

\providecommand*{\repst}{\mathrm{R}_{\dd}(\NsQ)^{\rm st}}
\providecommand*{\repsch}{\hat{\mathrm{R}}}
\providecommand*{\fl}{\mathrm{Fl}}

\providecommand*{\Gr}{\mathrm{Gr}}
\providecommand*{\afb}{\mathrm{RF}_\mathbf{d}} 
\providecommand*{\RF}{\mathrm{RF}_\mathbf{d}} 
\providecommand*{\para}{\mathrm{P}_{\mathbf{d}}}
\newcommand*{\pifib}[1][M]{\mathrm{Fl}_Q \binom{#1}{\mathbf{d}}^{(1)}}
\newcommand*{\FMd}[1][M]{\mathrm{Fl}_Q \binom{#1}{\mathbf{d}}}

\providecommand*{\op}{\mathrm{op}}
\providecommand*{\res}{\mathrm{res}}
\providecommand*{\inj}{\mathrm{inj}}

\providecommand*{\idt}{\mathrm{id}}

\providecommand*{\Ker}{\mathrm{Ker}\;}
\providecommand*{\Bild}{\mathrm{Im}\;}

\providecommand*{\Qs}{Q^{(s)}}

\providecommand*{\bil}{\text{-}}
\providecommand*{\iff}{\Longleftrightarrow}
\newcommand{\euler}[2]{\langle #1, #2 \rangle}
\providecommand*{\Span}{\mathrm{Span}}
\providecommand*{\Cospan}{\mathrm{Cosp}}
\providecommand*{\longar}{\xrightarrow{\makebox[2em]{}}}
\providecommand*{\st}{\mathrm{st}}
\providecommand*{\tell}{{\mathbb{L}}}

\newtheorem{thm}{Theorem}[section]
\newtheorem{lemma}[thm]{Lemma}

\newtheorem{prp}[thm]{Proposition}
\newtheorem{cor}[thm]{Corollary}
\theoremstyle{remark}
\newtheorem{remark}[thm]{Remark}

\newtheorem{example}[thm]{Example}
\newtheorem*{rem}{Remark}
\theoremstyle{definition}
\newtheorem{defn}[thm]{Definition}

\begin{document}

\title{Quiver-graded Richardson Orbits}

\author{\"Ogmundur Eir\'iksson}
\address{Fakult\"at f\"ur Mathematik, Universit\"at Bielefeld, Postfach 100 131, D-33501 Bielefeld, Germany}
\email{oeirikss@math.uni-bielefeld.de}

\author{Julia Sauter}
\address{Fakult\"at f\"ur Mathematik, Universit\"at Bielefeld, Postfach 100 131, D-33501 Bielefeld, Germany}
\email{jsauter@math.uni-bielefeld.de}




\thanks{Both authors are funded by the Humboldt Professorship of William Crawley-Boevey.}


\begin{abstract}
 In Lie theory, a dense orbit in the unipotent radical of a parabolic group under the adjoint action is called a Richardson orbit. 
We define a quiver-graded version of Richardson orbits generalising the classical definition in the case of the general linear group. 
In our setting a product of parabolic subgroups of general linear groups acts on a closed subvariety of the representation space of a quiver.
Such dense orbits do not exist in general. 
We define a quasi-hereditary algebra called the nilpotent quiver algebra whose isomorphism classes of $\Delta$-filtered modules correspond to orbits in our generalised setting.
We translate the existence of a Richardson orbit into the existence of a rigid $\Delta$-filtered module of a given dimension vector.
We study an idempotent recollement of this algebra whose associated intermediate extension functor can be used to produce Richardson orbits in some situations. 
This can be explicitly calculated in examples. 
We also give examples where no Richardson orbit exists.
\end{abstract}

\maketitle

\section{Introduction}

Let $G$ be a reductive group over an algebraically closed field $k$ of characteristic zero. Let $P$ be a parabolic subgroup of $G$ and $\mathfrak{u}_P$ the Lie algebra of its unipotent radical with the $P$-operation given by the adjoint action. 
Richardson proved in \cite{Ric} that $\mathfrak{u}_P$ has an open $P$-orbit, which we call a \emph{classical} Richardson orbit. 

For the group $\mathbb{G}=\GL_{n}$ we see $\mathbb{P}$ as a stabilizer of a flag $0=F^{(0)} \subset F^{(1)} \subset \cdots \subset F^{(s)} = k^{n}$, and $\mathfrak{u}_\mathbb{P}$ as the endomorphisms of $k^{n}$ mapping each subspace $F^{(t)}$ of the flag to $F^{(t-1)}$. 
To explain the idea of quiver-grading, we restrict (here in the introduction) to a finite, simply-laced quiver $Q$ with set of vertices $Q_0$ and set of arrows $Q_1$. 
We choose a diagonal Levi-subgroup $\GL_d := \prod_{i\in Q_0} \GL_{d_i} $ of $\GL_n$ and see 
$\rep_{d} := \prod_{(i\to j) \in Q_1} \Hom_k (k^{d_i}, k^{d_j})$ as a $\GL_d$-subrepresentation of $\rm{Lie} (\mathbb{G})=\rm{M}_n$, by picking the $d_j\times d_i$-matrix blocks in the matrix ring corresponding to $(j,i)\in Q_0^2$ whenever there is an arrow $(i \to j) \in Q_1$ in the quiver. 
We see this as the quiver-graded analogue of the adjoint representation (in type $A$). Now, $\para :=\mathbb{P}\cap \GL_{d}$ is a parabolic subgroup of $\GL_d$, i.e. it is the stabilizer of a $Q_0$-graded flag of subvector spaces $\{ 0 \} = F^{(0)} \subset \cdots  \subset F^{(s-1)} \subset F^{(s)}=\bigoplus_{i\in Q_0} k^{d_i}$, with $\Dim F^{(t)} =: \dd^{(t)}$, $\dd := (\dd^{(1)},\dots,\dd^{(s)}=d)$.  
A \emph{quiver-graded Richardson orbit} is then a dense $\para $-orbit in $\repdd:=\rep_{d}\cap \mathfrak{u}_{\mathbb{P}}$. 
More generally, for arbitrary finite quivers, a $Q_0$-graded flag as before is stabilized by a parabolic subgroup $\para$ in $\GL_d$. The flag also gives a subspace $\repdd := \{ (M_{i\to j}) \in \rep_d \mid M_{i\to j} (F_i^{(t)}) \subset F_j^{(t-1)}, t=1,\dots,s \}$ of the representation space, which has a $\para$-action by restricting the action of $\GL_d$.  
The associated fibre bundle $\GL_d \times^{\para} \repdd$ has a collapsing map $\pi_{\dd} \colon \GL_d \times^{\para} \repdd \to \rep_{d}, [g,M] \mapsto g\cdot M$. 
For acyclic quivers and complete flags this has been studied as a \emph{quiver-graded analogue} of the Springer map for $\GL_{d}$, and it leads to the geometric construction of KLR-algebras, cf. \cite{VV} and earlier work \cite{L}, \cite{R}. 
In reminiscence of the work of Richardson we say that the pair $(Q, \dd )$ \emph{has a Richardson orbit} if there is an open $\para$-orbit in $\repdd$. The connection to Lie theory when passing to a quiver-graded set-up is not entirely clear-but the quiver-graded Springer maps have been used to give a geometric construction of the positive part of the quantum group of the Kac-Moody Lie algebra associated to the quiver (see \cite{L}). 
We were inspired by the work of Br\"ustle-Hille-Ringel-R\"ohrle \cite{BHRR} who studied a finite-dimensional quasi-hereditary algebra to understand Richardson orbits from a representation theoretic side. 
They reproved the result of Richardson via an explicit construction, using ideas from \cite{HR}, where the Richardson orbits are shown to correspond to isomorphism classes of rigid $\Delta$-filtered modules over the Auslander algebra of $k[x]/\langle x^s \rangle$.
One advantage of the approach of \cite{BHRR} is that it provides explicit representatives for the orbits it parametrises. 
Similarly Baur \cite{Ba,Ba2} has constructed standard representatives for dense orbits in the nilpotent radical of parabolic subalgebras for classical Lie-algebras over $\C$,
this was extended to algebraically closed fields of characteristic greater than 2 by Baur-Goodwin \cite{BG}.
The idea to parametrise orbits by isomorphism classes of $\Delta$-filtered modules over an appropriate quasi-hereditary algebra, and such that rigid modules correspond to dense orbits, has also had further success;
Br\"ustle-Hille \cite{BH} extended the construction of \cite{BHRR} to all members of the descending central series of $\mathfrak{n}$,
Goodwin-Hille \cite{GH} constructed a quasi-hereditary algebra to parametrise certain Lie-ideals over the Borel subalgebra of $\mathfrak{gl}_n$, and Jensen-Su-Yu \cite{JSY} did similarly for orbits of the nilpotent radical of seaweed Lie-algebras in $\mathfrak{gl}_n$.

In Section \ref{sec:nsq} we follow the same strategy as above to describe the dense $\para$-orbits of $\repdd$.
To this end we construct an algebra $\NsQ$ for any finite quiver $Q$ and integer $s \in \N$, if $Q$ is the one loop quiver it is isomorphic to the Auslander algebra of $k[x]/\langle x^s \rangle$.
It has a left strongly quasi-hereditary structure given in terms of a layer function $\tell$ on the primitive idempotents as in \cite{Ri4}. It is also right ultra strongly quasi-hereditary in the sense of \cite{Co}.
The nilpotent quiver algebra also arises as a tensor algebra associated to an algebra and a bimodule in two different ways.
We give a bijection between quiver-graded Richardson orbits and the isomorphism classes of rigid $\Delta$-filtered $\NsQ$-modules in Theorem \ref{bijection}. 
Previous work \cite{BHRR} has shown that if $Q$ is the one-loop quiver then $\NsQ$ has only finitely many rigid indecomposable $\Delta$-filtered modules, 
but in general there may be infinitely many rigid indecomposable $\Delta$-filtered modules. 

The subcategory $\cF(\Delta)$ of $\Delta$-filtered $\NsQ$-modules can be embedded into the monomorphism categories as studied by Xiong-Zhang-Zhang in \cite{XZZ}. 
They generalise submodule categories which have been studied extensively, in particular by Ringel-Schmidmeier \cite{RiS,RiS2} and Simson \cite{Si,Si2}, although the studies go all the way back to Birkhoff \cite{Bi}. 
This gives a connection between $\afb$ and quiver flag varieties. In a different process, it is possible to embed quiver flag varieties of acyclic quivers in the fibres of $\pi_{\dd}$ for an appropriate $\dd$. 
These connections are outlined in Section \ref{subsec:moncat}.

In Section 4 we define the idempotent $e\in \NsQ$ to be the sum of the primitive idempotents in $\tell^{-1}(s)$, which we call the highest layer.
The idempotent subalgebra $e\NsQ e$ is isomorphic to the \emph{truncated path algebra} $kQ/J^s$, where $J$ is the two-sided ideal generated by the arrows in $Q$ cf. Lemma \ref{eNsQe}. The representation space of the truncated path algebra is our quiver-graded analogue of the nilpotent cone (i.e. the nilpotent matrices in $M_{n\times n} (k)$). Its irreducible components are parametrised in \cite{BHT}, we reobtain this result as a corollary from our results (cf. Corollary \ref{irrcompOfNilp}). 
We see the restriction of the functor $e\colon \NsQ \bil \modu \to kQ/J^s \bil \modu$ to $\cF(\Delta)$ as an algebraic version of the collapsing map $\pi_{\dd}$.
It is well-known that this functor has a left adjoint $\ell$, right adjoint $r$, and an intermediate extension functor $c= \Bild (\ell \to r )$, all fully faithful cf. \cite{Ku}. 
The functors $r$ and $c$ have their essential images in $\cF(\Delta)$.
For a dimension vector $d \in \N_0^{Q_0}$ we define $\rep_d(kQ/J^s):= \{ M \in \rep_d \mid J^s M =0\}$,
we prove our first main result in Section \ref{subsec:nsq recollement}.

\begin{restatable}{theorem}{existRO}
\label{existRO}
Let $Q$ be a quiver, $s\in \N$ and $d\in \N_0^{Q_0}$. For a given $M\in \rep_d(kQ/J^s)$ we take $\dd= \Dim c(M)$ (resp. $\dd = \Dim r(M)$). 
If $\Bild \pi_{\dd}$ has a dense $\GL_d$-orbit $\cO_N$, then $\pi_{\dd}$ is a desingularisation of $\overline{\cO_N}$ and $\repdd$ has a Richardson orbit contained in $\cO_N$.
This orbit is given by the rigid $\NsQ$-module $c(N)$ (resp. $r(N)$).
\end{restatable}


In particular, for every rigid $kQ/J^s$-module $M\in \rep_d$, we find Richardson orbits for $(Q, \Dim c(M))$ and $(Q, \Dim r(M))$, and the Richardson orbits are given by the $\NsQ$-modules $c(M)$ and $r(M)$ respectively. 
Also, whenever $kQ/J^s$ is representation-finite, there exist Richardson orbits for $(Q, \Dim c(M))$ and for $(Q, \Dim r(M))$ for any $M \in R_d(kQ/J^s)$.
We also prove that for $s=2$ the existence of a Richardson orbit is equivalent to the existence of a dense $\GL_d$-orbit in the image of $\pi_\dd$, cf. Proposition \ref{s2RO}.

Since we allow $Q$ to be an arbitrary finite quiver, the image of the collapsing map is often not an orbit closure, e.g. the example in Section \ref{subsec:Kronecker}, then there is no Richardson orbit.
Therefore we introduce a relaxed Richardson property which only requires that a generic fibre $\pi_{\dd}^{-1}(M)$ of the collapsing map has a dense orbit under the automorphism group $\aut_Q(M)$, cf. Definition \ref{DfnrelaxRP}. 
Then Theorem \ref{existRO} generalizes to the following. 

\begin{restatable}{theorem}{relaxed}
\label{thm:relax}
Let $Q$ be a finite quiver, $s\in \N$ and $d\in \N_0^{Q_0}$. For a given $M\in \rep_d(kQ/J^s)$ let either $\dd = \Dim c(M)$ or $\dd = \Dim r(M)$. \\
Then $\pi_{\dd}$ is birational onto its image and $(Q, \dd)$ has the relaxed Richardson property. 
\end{restatable}

 
We conclude with a zoo of examples in Section \ref{sec:examples}. 
In Section \ref{subsec:Auslander} we outline the overlap between Auslander algebras and $\NsQ$, which in particular shows how our theory applies to the classical case from \cite{BHRR}.
For the quiver $Q=A_2$ there exists a Richardson orbit for every $\dd$, and we provide an explicit algorithm to find the rigid $\Delta$-filtered module corresponding to that orbit in Section $\ref{subsec:algor}$. 



\section{Quiver-graded Richardson orbits}
\label{sec:QG-RO}

Let $k$ be an algebraically closed field. 
Let $Q = (Q_0,Q_1)$ be a finite quiver with vertices $Q_0$ and arrows $Q_1$, and $A=kQ$ be the path algebra, cf. \cite{ASS}. 
All varieties in this section with an $\rep$ in the name will depend on the quiver $Q$ and all group actions are assumed to be separated.

We fix a dimension vector $d=(d_i)_{i\in Q_0}\in \N_0^{Q_0}$, and denote the \emph{representation variety} (or representation space) of $Q$ for a dimension vector $d$ by 

\[\rep_d := \prod_{(i\to j)\in Q_1} \Hom_k (k^{d_i}, k^{d_j}) .\]
The reductive group $\GL_d := \prod_{i \in Q_0} \GL_{d_i} $ acts on $\rep_d$ by conjugation. 

Let $\dd=(\dd^{(1)}, \ldots , \dd^{(s)} = d)$ be a sequence of dimension vectors, where $\dd^{(t)}= (d_{i}^{(t)})_{i \in Q_0} \in \N_0^{Q_0}$ for $1 \leq t \leq s$, and $d^{(t-1)}_i\leq d^{(t)}_i $ for all $i \in Q_0$ and $1 \leq t\leq s$, we call this a \emph{dimension filtration} of $d$.
For each vertex $i \in Q_0$, fix a flag $F_i = ( F_i^{(1)} \subset \cdots \subset F_i^{(s)}=k^{d_i})$ such that $\dim F_i^{(t)} = d_i^{(t)}$ for $1 \leq t \leq s$. 
By convention we set $F_i^{(0)} = \{0 \} $. 
We say $F$ is a flag of dimension filtration $\dd$.
We denote by
\[ 
\para :=\prod_{i\in Q_0} \{g\in \GL_{d_i} \mid gF_i^{(t)} \subset F_i^{(t)}, \; 1\leq t\leq s\} \subset \GL_d
\]
the parabolic subgroup that stabilises the $Q_0$-graded flag $F:=\bigoplus_{i\in Q_0}F_i$ in  the $Q_0$-graded vector space $k^d:=\bigoplus_{i\in Q_0} k^{d_i}$. 
We also consider the closed subvariety
\[ 
\repdd := \{ (M_{i\to j})_{Q_1} \in \rep_d \mid M_{i\to j}(F_i^{(t)}) \subset F_j^{(t-1)}, 1\leq t\leq s  \} \subset \rep_d. 
\]
Note that this is a vector space, in particular it is smooth and irreducible. When we restrict the $\GL_d$-operation on $\rep_d$ to the subgroup $\para $, we can see $\repdd$ as a $\para $-subrepresentation of $\rep_d$. 

\begin{defn}
We say \emph{there is a quiver-graded Richardson orbit for }$(Q,\dd )$ if there is a dense $\para $-orbit in $\repdd$. In this case we call the dense $\para $-orbit the \emph{quiver-graded Richardson orbit}. 
\end{defn}

The variety $\repdd$ is a $\para $-invariant subvariety of $\rep_d$, we will consider the corresponding collapsing map
\begin{align*}
	\pi_{\dd}: \GL_d \times^{\para } \repdd \rightarrow \rep_d,
	 \quad [g,x] \mapsto g \cdot x.
\end{align*}

Since $\GL_d$ and $\repdd$ are smooth and irreducible the associated fibre bundle $\GL_d \times^{\para } \repdd$ is smooth and irreducible.
Note that $\pi_{\dd} $ is projective and $\GL_d$-equivariant, hence the image $\Bild\pi_{\dd}$ is a closed $\GL_d$-invariant subset of $\rep_d$. 
We identify $\GL_d/\para$ with the $Q_0$-graded flags of dimension filtration $\dd$ via the bijection $g \para \leftrightarrow gFg^{-1}$ and define 
\[
	\afb := \{ (M, U) \in \rep_d \times \GL_d/\para  \mid 
M_{i\to j}(U_i^{(t)}) \subset U_j^{(t-1 )}, \; \forall \: (i \rightarrow j) \in Q_1, \: 1 \leq t \leq s \}.
\]
This is a $\GL_d$-invariant subvariety of $\rep_d\times \GL_d/\para $ with the diagonal group action. 
If we apply \cite[Lemma 4, p.26]{Sl} to $\mathrm{pr}_2: \afb \rightarrow \GL_d/\para $ we get a $\GL_d$-equivariant isomorphism $\varphi \colon \GL_d \times^{\para } \repdd \to \afb$ given by $[g,M] \mapsto (g M g^{-1},g \para) $. Moreover the following diagram commutes 
\[
\begin{tikzcd}
\GL_d \times^{\para } \repdd \ar[rr,"\varphi"] \ar{dr}[swap]{\pi_{\dd}} &  & \afb \ar[dl,"\mathrm{pr}_1"] \\ 
& \rep_d. &
\end{tikzcd}
\]
Here we denote by $\FMd$ the variety of all flags of submodules $U =( U^{(1)}\subset \cdots \subset U^{(s)}=M)$ with $\Dim  U^{(t)}=\dd^{(t)}$. 
From the diagram above we see we can consider the fibre of $\pi_{\dd}$ over $M$ as embedded in a quiver flag variety.
We fix the following notation, where $U^{(0)} := \{ 0 \}$ 
\[ 
\pi^{-1}_{\dd} (M) =: \pifib =\{ U\in \FMd \mid U^{(t)}/ U^{(t-1)} \text{ semi-simple, } 1\leq t\leq s\}.
\]
This is a closed subvariety of $\FMd$. 
The stabilizer of $M$ with respect to the action of $\GL_d$ is $\aut_Q(M)$, which is defined as all elements of $\GL_d$ that are $kQ$-module endomorphisms of $M$. 

\begin{thm}
\label{thm:5equ} Let $(Q,\dd )$ and $\pi_{\dd}$ be as before. \\
Then, the following three statements are equivalent: 
\begin{itemize}
\item[$(1)$]
The variety $\repdd$ has a dense $\para $-orbit.
\item[$(2)$] The variety $\afb$ has a dense $\GL_d$-orbit.
\item[$(3)$] The variety $\Bild \pi_{\dd}$ has a dense $\GL_d$-orbit $\cO$ and for every point $M \in \cO$ 
the variety $\pifib$ has a dense $\aut_Q(M)$-orbit.
\end{itemize}
\end{thm}

For the proof we need the following lemma.

\begin{lemma}
\label{lem:1}
Let $H$ be a closed algebraic subgroup of an algebraic group $G$, and let $X$ be an irreducible $H$ variety.
The following are equivalent
\begin{itemize}
\item[$(1)$]
$G \times^H X$ has a dense $G$-orbit.
\item[$(2)$]
$X$ has a dense $H$-orbit.
\end{itemize}
\end{lemma}

\begin{proof}
We write $\phi\colon X \rightarrow G \times^H X, \; x \mapsto [1,x]$ for the canonical inclusion and identify $X$ with the image $\Bild \phi$.
Let $\cO$ be a $G$-orbit in $G \times^H X$, with $[g,x] \in \cO$.
Then $[1,x] = g^{-1}[g,x] \in \cO$, so any $G$-orbit in $G \times^H X$ has the form $G \cdot [1,x]$ for $x \in X$.
Assume $[g,y] \in \cO \cap X$. Then $[g,y]=[1,x]$ for some $x \in X$, i.e. there is $h \in H$ such that $(gh,h^{-1}y) = (1,x)$ for $x \in X$.
This shows $g \in H$, so $\cO \cap X \subset H \cdot [1,x]$. 
The inclusion $H \subset [1,x] \subset \cO \cap X$ is clear.

Assume $\cO = G \cdot [1,x]$ is a dense $G$-orbit of $G \times^H X$, then it is open. 
That implies $\cO \cap X$ is a non-empty open subset of $X$, and thus dense.

Conversely, assume $H \cdot [1,x] \subset X$ is dense. 
Now $H \cdot [1,x] \subset G \cdot [1,x]$ implies
\[
X = \overline{H \cdot [1,x]} \subset \overline{G \cdot [1,x]}.
\]
But all $G$-orbits intersect $X$, and the closure of $G \cdot [1,x]$ is $G$-invariant, thus it must be the union of all $G$-orbits of $G \times^H X$, which is of course all $G \times^H X$.
\end{proof}

\begin{proof}[Proof of Theorem \ref{thm:5equ}]
Conditions $(1)$ and $(2)$ are equivalent by Lemma \ref{lem:1}.

We show $(2)$ and $(3)$ are equivalent. Let $M \in \repdd$,
we have identities
\[
	\pi_{\dd}^{-1}(\GL_d M) = \GL_d \times^{\mathrm{stab}_{\GL_d}(M)} \pi_{\dd}^{-1}(M)
    = \GL_d \times^{\aut_Q(M)} \pifib.
\]
Thus Lemma \ref{lem:1} shows $\pi_{\dd}^{-1}(\GL_d \cdot M)$ has a dense $\GL_d$-orbit if and only if $\pifib$ has a dense $\aut_Q(M)$-orbit.

Now assume $\afb$ has a dense $\GL_d$-orbit $\GL_d \cdot [1,M]$. 
By the identity above that implies $\pifib$ has a dense $\aut_Q(M)$ orbit.
Moreover $\GL_d \cdot M$ is a dense $\GL_d$-orbit of $\Bild \pi_{\dd}$, because $\pi_{\dd}$ is continuous, thus condition $(3)$ holds.

Conversely, assume condition $(3)$ holds and let $M \in \cO$. 
Then $\pi_\dd^{-1}(\cO)$ is an open, and hence dense, subset of $\afb$.
Furthermore $\pi_{\dd}^{-1}(\cO)$ has a dense $\GL_d$-orbit because $\pifib$ has a dense orbit. 
But a dense orbit of a dense subset of $\afb$ is a dense orbit in $\afb$.

\end{proof}

\begin{example}
The motivating example for our setting comes from Lie theory. 
Let $Q$ be the one-loop quiver $\xymatrix{\bullet \ar@(ur,dr)^a }$ and $d \in \N$. 
Choose some flag $F$ on $k^d$, then $\para \subset \GL_d$ is just a classical parabolic subgroup determined by the flag.
In this case $\repdd$ coincides with the nilpotent radical $\mathfrak{n}$ of the Lie-algebra of $\para$, and the group action of $\para$ on $\repdd$ coincides with the adjoint action.
It is a classical result of Richardson \cite{Ric} that $\mathfrak{n}$ always has a dense $\para$ orbit.
\end{example}

It is straightforward to calculate the dimension of $\afb$ and $\rep_d$ in terms of $Q$ and $\dd$.
\begin{align*}
\dim \repdd &= \sum_{(i\to j) \in Q_1} \sum_{t=1}^s d_j^{t-1} ( d_i^t-d_i^{t-1})\\ 
\dim \RF &= \dim \repdd + \dim \GL_d/\para \\
 &=\sum_{(i\to j) \in Q_1} \sum_{t=1}^s d_j^{t-1} ( d_i^t-d_i^{t-1})+ \sum_{i\in Q_0} \sum_{r > t} (d_i^r-d_i^{r-1}) (d_i^t- d_i^{t-1})
\end{align*}
We set $d\cdot d := \sum_{i\in Q_0} d_i^2$. For reasons explained after Corollary \ref{cor:gl-dim}, we write
\[ 
\langle \dd , \dd \rangle^{(1)} := d\cdot d -\dim \RF.
\]

We have the following general easy lemma. We leave the proof to the reader (else cf. \cite[Lemma 39]{Sa}).

\begin{lemma} 
Let $G$ be a connected algebraic group, $H \subset G$ a closed subgroup and $V$ a $G$-variety with a smooth $H$-subvariety $F$. 
Assume $G \cdot F \subset V$ has a dense $G$-orbit $\cO$, 
then the fibres of $\pi \colon G\times^H F \to G \cdot F, \; [g,f] \mapsto gf,$ over $\cO$ are smooth, pairwise isomorphic, and irreducible of dimension $\dim G \times^H F - \dim \cO$.  
Furthermore, the following are equivalent 
\begin{itemize}
\item[$(1)$] The collapsing map $\pi \colon G\times^H F\to G \cdot F, \; [g,f] \mapsto gf$, is a resolution of singularities for $\overline{\cO}$. 
In other words $G \times^H F$ is irreducible and smooth, $G \cdot F = \overline{\cO}$, and $\pi$ is projective and an isomorphism over $\cO$.
\item[$(2)$]  $\pi^{-1}(v)\neq \emptyset$ for $v \in \cO$, and $\dim G\times^H F =\dim \cO$. 
\end{itemize}
\end{lemma}

This applies to our collapsing map $\pi_{\dd}\colon  \GL_d \times^{\para } \repdd \rightarrow \rep_d$ as follows:

\begin{cor} 
Let $M\in \rep_d $ and let $\dd$ be a dimension filtration of $d \in \N_0^{Q_0}$.
Assume that $\Bild \pi_{\dd} =\overline{\cO_M}$. Then the quiver flag varieties $\pifib[N]$ for $N\in \cO_M $ are pairwise isomorphic, smooth and irreducible of dimension 
\[ 
\dim \RF -\dim \rep_d + \dim \ext_Q^1 (M, M).
\]
Moreover the map $\pi_{\dd} \colon \RF \to \Bild \pi_{\dd} $ is a resolution of singularities of $\overline{\cO_M }$ if and only if the following two conditions are fulfilled: 
\begin{itemize}
\item[$(D1)$] 
$\pifib \neq \emptyset$;
\item[$(D2)$] 
$\dim \ext_Q^1 (M,M) =\dim \rep_d -\dim \RF \quad (\iff \; \dim \Hom_Q (M,M)  = \langle \dd ,\dd \rangle^{(1)})$.
\end{itemize}
In this case it follows that the restriction $\pi_{\dd}^{-1}(\cO_M )\to \cO_M $ is an isomorphism.
\end{cor}

Note that the conditions (D1) and (D2) in the corollary  imply that $\Bild \pi_{\dd}=\overline{\cO_M}$. 
If $\pi_{\dd} \colon \RF \to \Bild \pi_{\dd} $ is a resolution of singularities of an orbit closure, then the fibres over the dense orbit in $\Bild \pi_{\dd} $ consist only of a point. By Theorem \ref{thm:5equ} we conclude. 

\begin{cor} \label{desing}
If $\pi_{\dd}\colon \RF \to \Bild \pi_{\dd} $ is a resolution of singularities of an orbit closure, then 
there is a quiver-graded Richardson orbit for $(Q,\dd  )$. 
\end{cor}






\section{ The nilpotent quiver algebra}
\label{sec:nsq}

Our aim is to describe an algebra whose homological properties can be used to study the variety $\afb$, we fix a field $k$. 
Let $A$ be a finite dimensional algebra. We let $A \bil \fdmod$ (resp. $\fdmod \bil A$) denote the category of finite dimensional left (resp. right) $A$-modules. 
All modules considered are left modules unless explicitly stated otherwise.
We let $\kdual(-) = \Hom_k(-,k)$ denote the vector space duality,
it gives a contravariant functor $\kdual \colon \fdmod \bil A \to A \bil \fdmod$.


\subsection{The nilpotent quiver algebra}
\label{subsec:staircase quiver}

Let $Q = (Q_0,Q_1)$ be an arbitrary finite quiver and let $kQ$ be its path algebra. We write composition from the right to the left, i.e. the path 
$\xymatrix{ i \ar[r]^{\alpha} & j \ar[r]^{\beta} & l }$ 
is written $\beta \alpha$.
Thus representations of $Q$ can be seen as left $kQ$-modules.
Fixing $s \in \Z_+$ we define the \emph{staircase quiver} $\Qs$ of $Q$. 
It has vertices $i_t$ for $i \in Q_0$ and $t=1,\ldots ,s$.
It has two families of arrows, firstly there is an arrow $b(i_t) : i_t \rightarrow i_{t+1}$ for each $i \in Q_0$ and $t=1,....,s-1$. We call these arrows the \emph{vertical arrows}.
Also for each $(a \colon i \to j) \in Q_1$ and $t=2,...,s$ there is an arrow $(a_t \colon i_t \to j_{t-1})$, these arrows are called the \emph{diagonal arrows}. 
Consider the following relations of paths in $k\Qs$:
\begin{alignat}{2}
\label{eq:mesh2} \tag{R1}
	a_{t+1} b(i_t) &= b(j_{t-1}) a_{t}, \qquad &&  \forall (a \colon i \to j) \in Q_1, 1 < t < s,\\
 \label{eq:mesh3} \tag{R2}
	a_2 b(i_1) &= 0, && \forall (a \colon i \to j) \in Q_1.
\end{alignat}

We denote the equioriented quiver of type $A_s$ by $\A_s$, more precisely 
 \[
 	\A_s := \xymatrix{1 \ar[r] & 2 \ar[r] &  \cdots \ar[r] & s }.
\]
For clarification we draw the quiver $\A_n^{(s)}$. The relations are shown with dashed lines. Note that there are no relations in the top row.
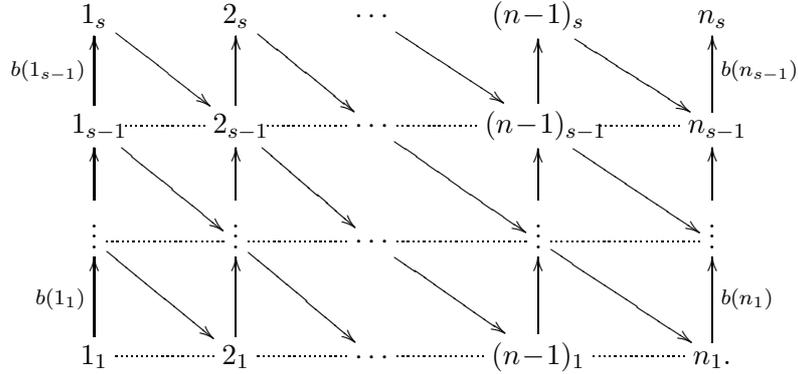
\[
\xymatrix@C = 2.7em{
	1_s \ar[rd] &  2_s \ar[rd] &  \cdots \ar[dr] & (n \!- \! 1)_s \ar[dr] & n_s \\
	1_{s-1} \!\! \ar[u]^{b(1_{s-1})} \ar[dr] \ar@{.}[r] & 2_{s-1} \!\! \ar[u] \ar[dr] \ar@{.}[r] & \cdots  \ar[dr] \ar@{.}[r] & (n \! - \! 1)_{s-1} \!\!\! \ar[u] \ar[dr] \ar@{.}[r] & n_{s-1} \!\! \ar[u]_{b(n_{s-1})} \\
	\vdots \ar[u] \ar[dr] \ar@{.}[r]&  \vdots  \ar[u] \ar[dr] \ar@{.}[r] & \cdots \ar[dr] \ar@{.}[r] & \vdots \ar[u] \ar[dr] \ar@{.}[r] & \vdots \ar[u] \\
	1_1 \ar[u]^{b(1_1)} \ar@{.}[r] &  2_1 \ar[u] \ar@{.}[r]  & \cdots  \ar@{.}[r] & (n \! - \! 1)_1 \ar[u] \ar@{.}[r] &  n_1. \ar[u]_{b(n_1)}
}
\]

We let $\mathfrak{I} \subset k \Qs$ be the ideal generated by the relations (\ref{eq:mesh2}) and (\ref{eq:mesh3}), and define the \emph{nilpotent quiver algebra} as $\NsQ := k \Qs / \mathfrak{I}$.

Every path in the path algebra $k\Qs$ can, up to the given relations, be written in a standard form.
Namely, if $\gamma$ is a path in $\Qs$ and not contained in $\mathfrak{I}$, then we have a unique path $\alpha$ of diagonal arrows and a unique path $\beta$ of vertical arrows such that $\gamma \in \beta \alpha + \mathfrak{I}$, note that both $\alpha$ and $\beta$ can be trivial.
All different paths of this form that are not in $\mathfrak{J}$ are linearly independent in the vector space $\NsQ$, so these form a basis.
We call it the standard basis of $\NsQ$.

\begin{example}
Note that if we take $Q$ to be the Jordan quiver, then $Q^{(s)}$ is the double of a quiver of type $A_s$.
\[
Q = \xymatrix{\bullet \ar@(ur,dr)^a }
\qquad
Q^{(s)} =
\xymatrix{
1 \ar@/^/[r]^{b_1} & \ar@/^/[l]^{a_2} 2 \ar@/^/[r]^{b_2} & \ar@/^/[l]^{a_3} \cdots \ar@/^/[r]^{b_{s-2}} & \ar@/^/[l]^{a_{s-1}} s \! - \! 1 \ar@/^/[r]^{b_{s-1}} & \ar@/^/[l]^{a_s} s.
}
\]
It is also immediate that the relations (\ref{eq:mesh2}) and (\ref{eq:mesh3}) are the same relations as those that define the Auslander algebra of $k[x]/\langle x^{s} \rangle$, cf. \cite[Section 5]{BHRR}.
Thus $\NsQ$ is isomorphic to said Auslander algebra. 
This was expected, our construction of $\NsQ$ is designed to expand the setting of \cite{BHRR}.
\end{example}


\subsection{Tensor algebras}
\label{subsec:Tensor algebra}

Let $A = \oplus_{n \geq 0} A_n$ be a positively graded algebra which is finitely generated over $A_0$, then $A_1$ is an $A_0 \bil A_0$-bimodule.
We have a tensor algebra
\begin{align*}
	T_{A_0}A_1 := \bigoplus_{n \geq 0} (\underbrace{A_1 \otimes_{A_0} \otimes \cdots \otimes_{A_0} A_1}_{n-\rm times}) .
\end{align*}
The multiplication is given by taking tensor products, and $T_{A_0}A_1$ has a natural $\Z$-grading where the $n$-th summand above has degree $n$.

The following is a general fact of graded algebras.
\begin{lemma}
\label{lem:tensor=graded}
Let $A \cong A_0 \langle x_1, \dots, x_n \mid r_1,\dots, r_m \rangle$, as a $\Z$-graded algebra.
Assume $x_i,r_j$ are homogeneous of degree 1 for all $i,j$.
Then $A$ is isomorphic to $T_{A_0}A_1$ as a graded algebra.
\end{lemma}

\begin{proof}
Since all the $r_j$ are homogeneous the grading on $A$ is well defined.
By the universal property of tensor algebras, the inclusion of $A_1$ in $A$ as an $A_0$-module induces a unique ring homomorphism $\phi \colon T_{A_0}A_1 \to A$. 
Now consider the inclusion map $\{ x_1, \dots,x_n \} \to A_1 \subset T_{A_0}A_1$.
This induces a module homomorphism $A_0 \langle x_1,\dots,x_n \rangle \to T_{A_0}A_1$, with all the $r_j$ in the kernel. Hence this induces a ring homomorphism $A \to T_{A_0}A_1$, which is inverse to $\phi$.
\end{proof}

For the rest of this subsection we assume $A \cong T_{A_0}A_1$ and we identify them as graded algebras.
However we only consider ungraded modules.
Let $A_+$ be the positively graded part of $A$ and let $M \in A \bil \fdmod$. 
The following lemma is an immediate consequence of a well known fact for tensor algebras, found for example in \cite[Prp 2.6, Chapter 2]{Cohn}.

\begin{lemma}
There is an exact sequence of $A$-modules
\begin{equation}
\label{eq:std-seq}
\tag{Std}
\begin{tikzcd}
0 \rar & A_+ \otimes_{A_0} M \rar["\delta_M"] & A \otimes_{A_0} M \rar["\epsilon_M"] & M \rar & 0.
\end{tikzcd}
\end{equation}
We call it the standard sequence.
Let $a \in A$ and $a_1 \in A_1$. The maps are given by
\begin{align*}
\epsilon_M(a \otimes m) & := a \cdot m,\\
\delta_M((a \otimes a_1) \otimes m) &:= (a \otimes a_1) \otimes m - a \otimes a_1 \cdot m.
\end{align*}
Dually, let $M$ be a right $A$-module. Then there is an exact sequence of (left) $A$-modules:
\[
\label{eq:std-seq2} \tag{DStd}
\begin{tikzcd}
	0 \rar & \kdual(M) \rar & \kdual(M \otimes_{A_0} A ) \rar & \kdual(M \otimes_{A_0} A_+ ) \rar & 0.
\end{tikzcd}
\]
\end{lemma}





\subsubsection{The nilpotent quiver algebra as a tensor algebra I}
\label{subsub:tensor-Lambda}

We put the following grading on $\NsQ$. 
We give each diagonal arrow of $\Qs$ the degree 1 and each vertical arrow the degree 0.
The relations (\ref{eq:mesh2}) and (\ref{eq:mesh3}) are homogeneous of degree 1 with respect to this grading, so it induces a grading on $\NsQ$. 
We let $\Lambda$ denote the resulting graded algebra.
Observe that $\Lambda$ satisfies the conditions of Lemma \ref{lem:tensor=graded}, thus $\Lambda \cong T_{\Lambda_0} \Lambda_1$ as a graded algebra. From now on we identify those algebras.

The algebra $\Lambda_0$ has the form $\Lambda_0 \simeq (k\A_s)^{Q_0}$, i.e. the disjoint union of $|Q_0|$ components, each isomorphic to the path algebra $k\A_s$.
Moreover $\Lambda_t = 0$ for $t \geq s$.


\subsubsection{The nilpotent quiver algebra as a tensor algebra II}
\label{subsub:tensor-Gamma}

We put the following grading on $\NsQ$. 
We give each vertical arrow of $\Qs$ the degree 1 and each diagonal arrow the degree 0.
The relations (\ref{eq:mesh2}) and (\ref{eq:mesh3}) are homogeneous of degree one with respect to this grading, so it induces a grading on $\NsQ$. 
We let $\Gamma$ denote the resulting graded algebra, where $\Gamma_r$ denotes the degree $r$ part of $\Gamma$.
Similarly as for $\Lambda$, the conditions of Lemma \ref{lem:tensor=graded} hold, so $\Gamma \cong T_{\Gamma_0}\Gamma_1$ as graded algebras.
From now on we identify $\Gamma$ with $T_{\Gamma_0}\Gamma_1$.

\begin{rem}
We have $\Gamma_t = 0$ for $t \geq s$.
The algebra $\Gamma_0$ is actually the path algebra of the subquiver that has only arrows of degree zero, in particular it is hereditary.
Even if $Q$ is connected, $\Gamma_0$ is usually not connected, in fact it has at least $|Q_0|$ components.
\end{rem}

\begin{lemma}
$\Gamma_{\Gamma_0}$ is projective as a right $\Gamma_0$-module.
Moreover both $\Gamma_1$ and $\Gamma$ are flat as right $\Gamma_0$-modules.
\end{lemma}

\begin{proof}
Now $\Gamma$ is isomorphic to the tensor algebra $T_{\Gamma_0} \Gamma_1$. Thus $\Gamma_{\Gamma_0}$ is flat if $\Gamma_1$ is flat as a right $\Gamma_0$-module.
Since projective modules are flat it suffices to show that $\Gamma_1$ is projective as right $\Gamma_0$-module. For this it is sufficient to show that $\Gamma_1$ is a right submodule of $\Gamma_0$, because $\Gamma_0$ is hereditary.

We know $\Gamma_0$ has a basis given by all non-trivial paths of the form $a_n \cdots a_1$, where the $a_m$ are diagonal arrows of $\Qs$.
Similarly $\Gamma_1$ has a basis given by all non-trivial paths in $\Qs$ of the form $b a_n \cdots a_1$, where $a_n \cdots a_1$ is a basis element of $\Gamma_0$, and $b$ is a vertical arrow such that $b a_n \neq 0$.
Clearly there exists at most one such arrow $b$.
Thus there is a well defined injective linear map
\begin{align*}
	\iota \colon \Gamma_1 \rightarrow \Gamma_0, \quad b a_n \cdots a_1 \mapsto a_n \cdots a_1.
\end{align*}
Let $a_0$ be a diagonal arrow, then
\begin{align*}
	\iota(b a_n \cdots a_1 a_0) = a_n \cdots a_1 \cdot a_0 = \iota(b a_n \cdots a_1) a_0.
\end{align*}
Hence $\iota$ is compatible with right multiplication by arrows of degree one.
It is also clearly compatible with right multiplication by trivial paths, thus $\iota$ is a homomorphism of right $\Gamma_0$-modules.
\end{proof}

\begin{prp}
\label{prp:les-ver}
Let $M,N$ be in $\Gamma \bil \fdmod$. We consider them as left $\Gamma_0$-modules where the context requires.
Then there is an exact sequence:
\[
\begin{tikzcd}
	0 \rar & \Hom_{\Gamma}(M,N) \ar[r] & \Hom_{\Gamma_0}(M,N) \ar[r] & 
\Hom_{\Gamma_0}(\Gamma_1 \otimes_{\Gamma_0} M,N)  \ar[out=0, in=180, looseness=1.5, overlay]{dll}  & \\
	& \ext^1_{\Gamma}(M,N) \ar[r] & \ext^1_{\Gamma_0}(M,N)  \ar[r] &
\ext^1_{\Gamma_0}(\Gamma_1 \otimes_{\Gamma_0} M,N) \ar[out=0, in=180, looseness=1.5, overlay]{dll}  \\
	 & \ext^2_{\Gamma}(M,N) \ar[r] & 0.
\end{tikzcd}
\]
\end{prp}

\begin{proof}
Let $P_{\bullet}$ be a projective resolution of $M$ as a $\Gamma_0$-module. 
Now $\Gamma_{\Gamma_0}$ is flat so $\Gamma \otimes_{\Gamma_0} P_{\bullet}$ is a projective resolution of $\Gamma \otimes_{\Gamma_0} M$.
Also observe that $\Hom_{\Gamma}(\Gamma,N) \cong N$ as $\Gamma_0$-modules.
We obtain the following identities by the hom-tensor adjunction:
\begin{align*}
	\ext^n_{\Gamma}(\Gamma \otimes_{\Gamma_0} M,N) 
	\cong H^n \Hom_{\Gamma} (\Gamma \otimes_{\Gamma_0} P_{\bullet},N)
	\cong H^n \Hom_{\Gamma_0} (P_{\bullet} , \Hom_{\Gamma}(\Gamma,N))
	\cong \ext^n_{\Gamma_0}(M,N).
\end{align*}
Similarly we get the identity:
\begin{align*}
	\ext^n_{\Gamma}(\Gamma \otimes_{\Gamma_0} \Gamma_1 \otimes_{\Gamma_0} M,N) 
	&\cong \ext^n_{\Gamma_0}(\Gamma_1 \otimes_{\Gamma_0} M, N).
\end{align*}
Apply those identities to the long exact sequence obtained by applying $\Hom_{\Gamma}(-,N)$ to the exact sequence (\ref{eq:std-seq}). 
Since $\Gamma_0$ is hereditary we have 
\[
\ext^n_{\Gamma_0}(M,N)=0=\ext^n_{\Gamma_0}(\Gamma_1 \otimes_{\Gamma_0} M, N), \quad n \geq 2,
\] 
hence we have an exact sequence of the form stated.
\end{proof}

The exact sequence implies all $n$-th extensions vanish for $n > 2$, hence we get:

\begin{cor}
\label{cor:gl-dim}
The algebra $\NsQ$ has global dimension at most 2.
\end{cor}

If $A$ is a finite-dimensional algebra of finite global dimension, then the Euler form for $A$-modules $M$ and $N$ is defined as 
\[ 
\langle M, N\rangle_A := \sum_{i=0}^{\infty } (-1)^i \dim_k \ext^i_A(M,N).
\]
It is well known that this gives a bilinear form on the Grothendieck group of $A \bil \modu$, i.e. it is determined by the dimension vectors of $M$ and $N$. 
We realised $\NsQ$ as the path algebra $k\Qs$ modulo the ideal $\mathfrak{I}$. 
The generators (\ref{eq:mesh2}) and (\ref{eq:mesh3}) of $\mathfrak{I}$ can be seen as extra arrows $(i_t \to j_t)$ for each $t = 1,\dots,s-1$ and $(i \to j) \in Q_1$, we denote those arrows by $\Qs_2$. 
Since $\NsQ$ has global dimension at most 2 a result of Bongartz \cite{B} implies the Euler form can be written as
\begin{align*}
	\euler{\dd}{\ee}_{\NsQ} = \sum_{i_t \in \Qs_0} \dd_{i_t} \ee_{i_t} 
	- \sum_{(i_t \to j_{u}) \in \Qs_1} \dd_{i_t}\ee_{j_{u}} + \sum_{(i_t \to j_t) \in \Qs_2} \dd_{i_t}\ee_{j_t}.
\end{align*}

\subsection{Monomorphism categories and the category $\cN$}
\label{subsec:moncat}

We are interested in $\NsQ$-modules that are related to Richardson orbits.

\begin{defn}
The category $\cN$ is the full subcategory of $\NsQ \bil \fdmod$ given by $\Qs$ representations that satisfy the relations (\ref{eq:mesh2}) and (\ref{eq:mesh3}) with the property that all maps corresponding to vertical arrows are injective.
\end{defn}

Let $T_s(Q)$ denote the algebra of lower triangular $s \times s$ matrices with coefficients in the path algebra $kQ$.
The category $T_s(Q) \bil \fdmod$ is equivalent to the following category.
It has as objects $s$-tuples of $kQ$-modules $M = (M_1,\dots,M_s)$ along with 
$kQ$-module homomorphisms $\phi_t: M_t \rightarrow M_{t+1}$ for $t=1,\dots,s-1$.
A morphisms $f:M \rightarrow M'$ is a tuple $(f_1,\dots,f_s)$ of morphisms of $Q$-representations compatible with the maps $\phi_t$ and $\phi'_t$.

The \emph{monomorphism category} $\mathrm{mon}_s(Q)$ is the full subcategory of objects such that $\phi_1,\dots,\phi_{s-1}$ are all monomorphisms.
These monomorphism categories coincide with those in \cite{XZZ}. They generalise the better known submodule categories, whose studies essentially go back to Birkhoff \cite{Bi}, although the name is newer.
The subcategory $\cN$ can be considered as a nilpotent analogue to the monomorphism categories 

For a path $\alpha$ in $kQ$ we let $\alpha(q,t)$ denote the matrix in $T_s(Q)$ that has $\alpha$ in coordinate $(q,t)$ and all other coordinates zero.
\begin{lemma}
There is a ring homomorphism $\Phi \colon T_s(Q) \rightarrow N_s(Q)$, determined by the following data: 

\begin{align*}
\Phi(a(t,t)) &= \begin{cases} 0, \; \; \; \quad \quad \quad & t=1,  \\ 
b(j_{t-1}) a_t,                  \; \; \;\quad \quad \quad & t=2,\dots,s. 
\end{cases}
\quad \forall (a \colon i \to j) \in Q_1, \\
\Phi(e_i(q,t)) &= 
\begin{cases} e(i_t),      \quad & q=t, \\
b(i_{q-1})\cdots b(i_t), \;\quad & q > t. \end{cases}
\quad \forall i \in Q_0.
\end{align*}
\end{lemma}

\begin{proof}
Clearly the elements whose value is determined above generate $T_s(Q)$ as a $k$-algebra, so there is at most one $k$-algebra homomorphism satisfying those conditions.
It is easy to check the homomorphism conditions on the generators $e_i(q,t)$, namely
\begin{align*}
	\Phi(e_i(q,t)) \Phi(e_j(q',t') ) 
	&= b(i_{q-1})\cdots b(i_t) \cdot b(j_{q'-1})\cdots b(j_{t'})\\
	&= \delta_{ij} \delta_{q't} b(i_{q-1}) \cdots b(j_{t'}) \\
	&= \Phi( e_i(q,t) e_j(q',t') ).
\end{align*}

We also have to show that the relations of $T_s(Q)$ are sent to zero.
Firstly, if $a \colon i \to j$ and $a' \colon i' \to j'$ are arrows of $Q_1$ and if $a a' = 0$ or $t \neq q$, then $\Phi(a(t,t)) \Phi(a'(q,q)) = b(j_{t-1})a_t b(j'_{q-1})a'_q = 0$.
Finally observe that
\begin{align*}
	\Phi( a(t,t) e_i(t,t\!- \! 1) ) - \Phi(e_j(t,t\!- \! 1) a(t \!- \!1,t \!- \! 1)) 
	&=  b(j_{t-1}) a_t b(i_{t-1}) -  b(j_{t-1}) b(j_{t-2}) a_{t-1}\\
    &= b(j_{t-1}) (a_t b(i_{t-1}) -  b(j_{t-2}) a_{t-1}) \in \mathfrak{I}.
\end{align*}
\end{proof}

There is a restriction functor $\Phi^*: N_s(Q) \bil \fdmod \rightarrow T_s(Q) \bil \fdmod$ induced by $\Phi$.

\begin{prp}
The functor $\Phi^*$ restricts to a fully faithful functor from $\cN$ to $\mathrm{mon}_s(Q)$.
\end{prp}

\begin{proof}
The functor $\Phi^*$ is faithful because it is a restriction functor.
Let $N$ be an object of $\cN$.
Then $x \mapsto \Phi(e_i(q,t)) x = b(i_{q-1})\cdots b(i_t) x$ is an injective linear map $e_{i_t}N \to e_{i_q}N$ for all $i \in Q_0$ and $1 \leq t < q \leq s$ because $N \in \cN$. 
But we can characterise $\mathrm{mon}_s(Q)$ as the full subcategory of $T_s(Q) \bil \fdmod$ with objects $M$ such that 
$x \mapsto e_i(q,t)x$ is a monomorphism of vector spaces $e_{i}(q,q)M \rightarrow e_i(t,t)M$ for all $i \in Q_0$ and $1 \leq t \leq q \leq s$, thus $\Phi^*(N) \in \mathrm{mon}_s(Q)$.

Let $N,N' \in \cN$ and let  $f \in \Hom_{T_s(Q)}(\Phi^*(N),\Phi^*(N'))$.
We know $f$ is compatible with multiplication with all elements in the image of $\Phi$.
The image contains all the trivial paths and all paths that have only vertical arrows.
It only remains to show that $f$ is compatible with $a_t$ for all $a \in Q_1$ and $t=2,\dots,s$.
Let $m \in N$, since $\Phi(a(t,t)) = b(j_{t-1}) a_t$ we know
\[
	b(j_{t-1}) a_t \cdot f(m) = f(b(j_{t-1}) a_t \cdot m) = b(j_{t-1}) f(a_t \cdot m).
\]
But multiplication by $b(j_{t-1})$ on $e_{i_t}N'$ is a monomorphism of vector spaces by assumption, which implies $f(a_t \cdot m) =  a_t \cdot f(m)$. 
Thus $f$ is compatible with multiplication by a set of generators of $N_s(Q)$.
\end{proof}

\begin{remark}
The essential image of the restriction of $\Phi^*$ to $\cN$ can be characterised as the full subcategory of objects $M$ in $\mathrm{mon}_s(Q)$ such that the quotient $M_{t+1}/M_t$ is a semi-simple $kQ$-module for $t=1,\dots,s-1$.
\end{remark}







\subsubsection{The category $\cN$ and tensor algebra structure}
\label{subsub:N and tensor algebra}

We give other characterisations of the modules in $\cN$. For a tensor algebra $\Gamma =T_{\Gamma_0 } (\Gamma_1 )$ the category of left $\Gamma $-modules is equivalent to a category with objects pairs $(M, \varphi )$, where $M$ is a $\Gamma_0$-module and $\varphi \colon \Gamma_1\otimes_{\Gamma_0} M \to M$ is a $\Gamma_0$-linear map. 
The morphisms from $(M, \varphi ) $ to $(N, \psi )$ in this category are given by a $\Gamma_0$-homomorphism $f\colon M\to N$ such that the following diagram commutes 
\[ 
\begin{tikzcd} 
\Gamma_1\otimes_{\Gamma_0} M \ar[rr,"\varphi"] \ar{d}[swap]{\idt \otimes f} && M \ar[d,"f"] \\
\Gamma_1\otimes_{\Gamma_0} N \ar[rr,"\psi"] && N.
\end{tikzcd}
\]
We denote this category by $(\Gamma_0,\Gamma_1) \bil \fdmod$.
The equivalence is given by restricting the $\Gamma $-module structure on $M$ to the $\Gamma_0$-module structure, which we denote by ${}_{\Gamma_0} M$, and restricting the scalar multiplication $\Gamma\otimes_{\Gamma}M\to M$ to $\Gamma_1 \otimes_{\Gamma} M$, to obtain the map $\varphi $. 

Recall that $\NsQ$ arises as a tensor algebra in two different ways, as $\Lambda$ (cf. Section \ref{subsub:tensor-Lambda}) and as $\Gamma$ (cf. Section \ref{subsub:tensor-Gamma}). Here we use the latter construction.

\begin{prp}
\label{prp:N=monos}
The full subcategory of pairs $(M,\varphi)$  in $(\Gamma_0,\Gamma_1) \bil \fdmod$  such that $\varphi$ is a monomorphism corresponds to $\cN$ under the equivalence above.
\end{prp}

\begin{proof}
Let $(M,\varphi)$ be an object of $(\Gamma_0,\Gamma_1) \bil \fdmod$ and let $m \in M$.
Then $b(i_t)\otimes m = b(i_t) \otimes e(i_t)m$, 
so $\varphi(b(i_t) \otimes m) = b(i_t) e(i_t) \cdot m = 0$ if and only if $b(i_t) m= 0$.  
Thus $\varphi$ is injective if and only if the map $e(i_t) M \rightarrow e(i_{t+1})M$, given by left multiplication by $b(i_t)$, is injective for all $t = 1,\dots,s-1$ and all $i \in Q_0$.
This shows $(M,\varphi)$ corresponds to an object in $\cN$ if and only if $\varphi$ is a monomorphism.
\end{proof}


\subsubsection{Embedding monomorphism categories  in $\cN$}

Consider the following example.
Take the quiver $Q = \xymatrix{1 \ar[r] & 2 & \ar[l] 3}$. We can embed the quiver of $T_2(Q)$ into $\Qs$ as the full subquiver of bold dots in the following diagram: 
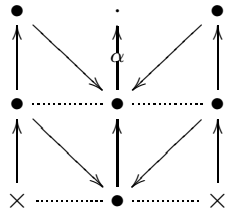
\[
\xymatrix{
\bullet \ar[dr] & \cdot & \ar[dl] \bullet \\
\bullet\ar[dr] \ar[u] \ar@{.}[r] & \bullet \ar[u]|{\alpha} & \ar@{.}[l] \ar[dl] \bullet \ar[u] \\
\times \ar@{.}[r] \ar[u] & \bullet \ar[u] & \times \ar@{.}[l] \ar[u]
}
\]
Observe that the commutativity relations in $Q^{(3)}$ give the commutativity relations of the subquiver of bold dots, so we get the quiver with relations of $T_2(Q)$. 
By adding the identity for the arrow labelled by $\alpha$, and placing the zero vector space on the vertices marked $\times$, one can view an object of $\mon_2(Q)$ as an object in $\cN \subset N_3(Q) \bil \modu$. 

More generally, if $Q$ is acyclic with all paths of length $\leq n$, then one can embed the monomorphism category $\mon_{s-n}(kQ)$ in $\cN \subset \NsQ \bil \fdmod$. 
On the geometrical side this means we can embed any quiver flag variety of an acyclic quiver in a fibre of the collapsing map $\pi_{\dd}$ for some dimension filtration $\dd$. 
Since every projective variety arises as a quiver Grassmannian of an acyclic quiver, this demonstrates that these fibres can be rather complicated,
and general questions about orbits in $\afb$ can be expected to be impossibly difficult.

Be aware that this embedding is not an inverse to the functor $\Phi^*$ in any sensible way.


\subsection{Quasi-hereditary structure}
\label{subsec:QH}

Let $A$ be a finite dimensional algebra and let $S(A)$ denote the set of isomorphism classes of simple left $A$-modules.
For $i \in S(A)$, we let $_AP(i)$ denote the projective cover of $i$, and $_AI(i)$ the injective envelope of $i$ in $A \bil \fdmod$.
We write $P(i)$ and $I(i)$ when it is clear in which category we take projective cover or injective envelope.
We let $S(A)_A$ denote the simple right $A$-modules up to isomorphism.
For $i \in S(A)_A$,
we let $P(i)_A$ denote the projective cover and $I(i)_A$ the injective envelope in $A \bil \modu$.


\subsubsection{Strongly quasi-hereditary algebras}

Let us recall the notion of a \emph{left-strongly hereditary} algebra as introduced in \cite[Section 4]{Ri4}.
We say $A$ has a left strongly quasi-hereditary structure if there is a layer function $\tell: S(A) \rightarrow \N$ such that for any simple $i$ there is an exact sequence
\[
\begin{tikzcd}
0 \rar & P_1 \rar & P(i) \rar & \Delta(i) \rar & 0,
\end{tikzcd}
\]
satisfying the following properties:
\begin{itemize}
\item[(LS1)] 
The module $P_1$ is projective and all indecomposable direct summands $P(j)$ of $P_1$ are such that $\tell(j) < \tell(i)$.
\item[(LS2)]
If $j$ is a composition factor of the radical of $\Delta(i)$, then $\tell(j) > \tell(i)$.
\end{itemize}
We call the modules $\Delta(i)$ the \emph{standard modules} of $A$. 

Dually, for each $i \in S(A)$ we let $\nabla(i)$ denote the maximal submodule of $I(i)$ such that for all of its composition factors $j \in S(A)$, we have $\tell(j) > \tell(i)$ or $j = i$. We call the modules $\nabla (i)$ the \emph{costandard modules} of $A$. 

\begin{rem}
Notice that we use a different orientation on the layer function than \cite{Ri4}.
\end{rem}

It is shown in \cite[Section 4]{Ri4} that the standard modules are in fact standard modules of a quasi-hereditary structure on $A$.
In fact a left strongly quasi-hereditary structure is equivalent to a quasi-hereditary structure such that all the standard modules have projective dimension at most one.
In this case the standard module $\Delta(i)$ can be characterised as the maximal factor module of $P(i)$ that satisfies condition (LS2).
Similarly a quasi-hereditary algebra is \emph{right strongly quasi-hereditary} if all the costandard modules have injective dimension at most one.
We denote by $\cF (\Delta )$ (resp. $\cF (\nabla )$) the full subcategory of $A \bil \fdmod$ of objects that have a filtration of standard (resp. costandard) modules.
The \emph{characteristic tilting module} of a quasi-hereditary structure is the unique basic module $T$ such that $\add(T) = \cF(\nabla) \cap \cF(\Delta)$. 

In \cite{Co}, Conde introduced the notion of \emph{right ultra strongly quasi-hereditary} algebras.
Those are quasi-hereditary algebras that satisfy the conditions
\begin{itemize}
\item[(US1)] 
$\rad(\Delta(i))$ is either a standard module or zero.
\item[(US2)] 
If $\rad(\Delta(i)) =0$, then $I(i) \in \cF(\Delta)$.
\end{itemize}
These algebras are in particular right strongly quasi-hereditary.


\subsubsection{Quasi-hereditary structure of the nilpotent quiver algebra}
\label{subsub:QH}

Keep in mind that the categories $\Gamma \bil \fdmod, \Lambda \bil \fdmod$ and $\NsQ \bil \fdmod$ are all equivalent.
The sets $S(\Gamma),S(\Lambda),S(\Gamma_0),S(\Lambda_0),$ 
$S(\Gamma)_{\Gamma},S(\Lambda)_{\Lambda},S(\Gamma_0)_{\Gamma_0}$ and $S(\Lambda_0)_{\Lambda_0}$ are all in a canonical bijection with the vertices of $\Qs$.
Accordingly, if $A = \Gamma,\Gamma_0,\Lambda$ or $\Lambda_0$, we write $_AS(i_t)$ (resp. $S(i_t)_A$) for the simple left (resp. right) $A$-module corresponding to $i_t$, and $_AP(i_t),{_AI}(i_t)$ (resp. $P(i_t)_A,I(i_t)_A$) for the projective cover and injective envelope of that simple.
We write $P(i_t) = {_{\Lambda}P}(i_t) \cong {_{\Gamma}P}(i_t)$ and similarly $I(i_t) = {_{\Lambda}I}(i_t) \cong {_{\Gamma}I}(i_t)$ as our default setting.
Let us consider the layer function
\[
\tell \colon S(\NsQ) \to \N, \quad \tell(i_t) := \tell(S(i_t)) := t.
\]

If $e \in \NsQ$ is an idempotent, we denote the two-sided ideal generated by $e$ by $(e)$.
For $i \in Q_0$ we define $\Span(i) := \{ (a \colon i \to j) \in Q_1\}$, and dually $\Cospan(i) := \{ (a \colon j \to i) \in Q_1\}$. 

\begin{lemma}
\label{lem:+tensor}
Let $t \in \{2,\dots,s\}$.
Then
\[
\tag{Syz}
\label{eq:PM1}
\Lambda_+ \otimes_{\Lambda_0} {_{\Lambda_0}P}(i_t) \cong \bigoplus_{\substack{(a \colon i \to j) \\ \in \Span(i)}} {_{\Lambda}P}(j_{t-1}).
\]

Also
\[
\tag{Cosyz}
\label{eq:PM2}
	P(i_{t})_{\Gamma_0} \otimes_{\Gamma_0} \Gamma_+ \cong P(i_{t-1})_{\Gamma}.
\]
Let $e_1 = \sum_{i \in Q_0} e(i_1)$,
there is a canonical inclusion $\iota \colon \fdmod \bil \Lambda/(e_1) \to \fdmod \bil \Lambda$. 
Then we have the following identity of right $\Lambda$-modules.
\[
\tag{Emb}
\label{eq:PM3}
	P(i_s)_{\Lambda_0} \otimes_{\Lambda_0} \Lambda_+ \cong \bigoplus_{\substack{(a \colon j \to i) \\ \in \Cospan(i)}} \iota (P(j_s)_{\Lambda/(e_1)}).
\]
\end{lemma}

\begin{proof}
The first identity (\ref{eq:PM1}) can be seen from the following identities of $\Lambda_0$-modules.
\begin{align*}
	\Lambda_1 \otimes_{\Lambda_0} {_{\Lambda_0}P}(i_t)
	&= \bigoplus_{\substack{(a \colon i \to j) \\ \in \Span(i)}} \Lambda_0 a_{t-1}  e(i_t) 
	= \bigoplus_{\substack{(a \colon i \to j) \\ \in \Span(i)}} \Lambda_0 e(j_{t-1}) = \bigoplus_{\substack{(a \colon i \to j) \\ \in \Span(i)}} {_{\Lambda_0}P}(j_{t-1}).
\end{align*}
Since $\Lambda_+ = \Lambda \otimes_{\Lambda_0} \Lambda_1$ we get the first identity from the statement applying $\Lambda \otimes_{\Lambda_0} -$ to the identity above.

In a similar way we get
\begin{align*}
P(i_{t})_{\Gamma_0} \otimes_{\Gamma_0} \Gamma_1 
\cong e(i_t) \Gamma_0 \otimes_{\Gamma_0} \Gamma_1
\cong e(i_t) \Gamma_1
\cong b(i_{t-1}) \Gamma_0 \cong P(i_{t-1})_{\Gamma_0}.
\end{align*}
Applying $- \otimes_{\Gamma_0} \Gamma$ to this gives the identity (\ref{eq:PM2}).

Finally we prove (\ref{eq:PM3}): First observe that we have the following identity of right $\Lambda_0$-modules.
\[
e(i_{s})\Lambda_1 \cong \bigoplus_{\substack{(a \colon j \to i) \\ \in \Cospan(i)}} e(j_{s}) \Lambda_0/(e_1).
\]
That gives the following isomorphisms of right $\Lambda$-modules:
\begin{align*}
P(i_s)_{\Lambda_0} \otimes_{\Lambda_0} \Lambda_+ &\cong e(i_s) \Lambda_1 \otimes_{\Lambda_0} \Lambda
\cong \bigoplus_{\substack{(a \colon j \to i) \\ \in \Cospan(i)}} e(j_{s}) \Lambda_0/(e_1) \otimes_{\Lambda_0} \Lambda
\cong \bigoplus_{\substack{(a \colon j \to i) \\ \in \Cospan(i)}} \iota (P(j_s)_{\Lambda/(e_1)}).
\end{align*}


\end{proof}

Let $\Delta(i_t)$ denote the maximal factor module of $P(i_t)$ such that all the composition factors $S(j_u)$ of $\rad \Delta(i_t)$ have layer $\tell(j_u) > t$.

\begin{prp}
\label{prp:seqS1S2}
Let $t \in \{ 2,\dots,s\}$.
There is an exact sequence
\[
\label{seq:proj} \tag{Res}
0 \longar  \bigoplus_{\substack{(a \colon i \to j) \\ \in \Span(i)}} P(j_{t-1}) \longar   P(i_t) \longar  \Delta(i_t) \longar  0
\]
of $\NsQ$-modules. This sequence is a projective resolution of $\Delta(i_t)$ and satisfies $(LS1)$ and $(LS2)$. 
Also, the costandard modules have the following injective coresolution.
\[ \label{seq:inj} \tag{Cores}
\begin{tikzcd}
0 \rar & \nabla(i_t) \rar & I(i_t) \rar & I(i_{t-1}) \rar & 0.
\end{tikzcd}
\]
\end{prp}

\begin{proof}
Apply the sequence (\ref{eq:std-seq}) to the $\Lambda$-module $\Lambda/\Lambda_+ \otimes_{\Lambda_0} {_{\Lambda_0}P}(i_t)$.
Using the identity (\ref{eq:PM1}) we get the sequence
\[
0	\longar \bigoplus_{\substack{(a \colon i \to j) \\ \in \Span(i)}} P(j_{t-1})  \longar P(i_t) \longar \Lambda/\Lambda_+ \otimes_{\Lambda_0} {_{\Lambda_0}P}(i_t) \longar 0.
\]
Now observe that all the composition factors of $\Lambda/\Lambda_+ \otimes_{\Lambda_0} {_{\Lambda_0}P}(i_t)$ have layer higher than or equal to $t$. Thus $\Lambda/\Lambda_+ \otimes_{\Lambda_0} {_{\Lambda_0}P}(i_t)$ is a factor module of $\Delta(i_t)$.
On the other hand the top of each summand of the kernel of the sequence has layer $t-1 < t$, hence $\Delta(i_t)$ is a factor module of $\Lambda/\Lambda_+ \otimes_{\Lambda_0} {_{\Lambda_0}P}(i_t)$. 
Together this shows $\Lambda/\Lambda_+ \otimes_{\Lambda_0} {_{\Lambda_0}P}(i_t) \cong \Delta(i_t)$.

Next we apply the sequence (\ref{eq:std-seq2}) to the right $\Gamma$-module $P(i_t)_{\Gamma_0} \otimes_{\Gamma_0} \Gamma/\Gamma_+ $. By the identity (\ref{eq:PM2}) we get the following sequence of left $\Gamma$-modules.
\[ 
\begin{tikzcd}
0 \rar & \kdual(P(i_t)_{\Gamma_0} \otimes_{\Gamma_0} \Gamma/\Gamma_+) \rar & I(i_t) \rar & I(i_{t-1}) \rar & 0.
\end{tikzcd}
\]
Now $\kdual( P(i_t)_{\Gamma_0} ) \cong {_{\Gamma_0}I}(i_t)$, 
and thus $\kdual(P(i_t)_{\Gamma_0} \otimes_{\Gamma_0} \Gamma/\Gamma_+) \cong \Gamma/\Gamma_+ \otimes_{\Gamma_0} {_{\Gamma_0}I}(i_t)$.
Observe that all composition factors of $\Gamma/\Gamma_+ \otimes_{\Gamma_0} {_{\Gamma_0}I}(i_t)$ have layer $\geq t$.
Moreover, any submodule of $I(i_t)$ that properly contains this submodule, would have $S(i_{t-1})$ as composition factor.
Together this shows that $\Gamma/\Gamma_+ \otimes_{\Gamma_0} {_{\Gamma_0}I}(i_t) \cong \nabla(i_t)$, the costandard module at $i_t$.
\end{proof}

Let $\add(\Delta)$ (resp. $\add(\nabla)$) denote the full subcategory of $\NsQ \bil \fdmod$ given by finite direct sums of standard modules (resp. costandard modules).

\begin{cor}
\label{cor:equivalence}
The algebra $\NsQ$ is left strongly quasi-hereditary.
Moreover there are equivalences of categories
\[
\begin{tikzcd}
\Lambda_0 \bil \proj \ar["\Lambda/\Lambda_+ \otimes_{\Lambda_0} -",rr,shift left=1ex] & & \ar["\res_{\Lambda_0}",ll,shift left=1ex] \add(\Delta), & & 
\Gamma_0 \bil \inj\ar["\Gamma/\Gamma_+ \otimes_{\Gamma_0} -",rr,shift left=1ex] & & \ar["\res_{\Gamma_0}",ll,shift left=1ex] \add(\nabla).
\end{tikzcd}
\]
Here $\res_{\Lambda_0}$ and $\res_{\Gamma_0}$ are the restriction functors induced by the inclusions $\Lambda_0 \to \Lambda$ and $\Gamma_0 \to \Gamma$ respectively.
\end{cor}

For $t=1,\dots,s$ we define $e_t := \sum_{i \in Q_0} e(i_t)$, and $E_t := \sum_{u=1}^t e_{u}$, and $E_0 := 0$.
Let us write $( E_t) $ for the two sided ideal of $\NsQ$ generated by the idempotent $E_t$.
For each $t$ we have the quotient $\NsQ \to \NsQ/ (E_t)$, observe that $\NsQ/ (E_t) \cong N_{s-t}(Q)$.
The quotient induces a fully faithful functor $\iota_t \colon \NsQ/ (E_t) \bil \fdmod \to \NsQ \bil \fdmod$.
For $t = 0,\dots,s-1$ we define $T(i_{t+1}) := \iota_{t} (  {_{N_{s-t}(Q)}I}(i_s))$.

\begin{prp}
\label{prp:char-tilt}
There are sequences 
\[
\label{seq:char} \tag{Filt}
0 \longar \Delta(i_{t}) \longar T(i_{t}) \longar \bigoplus_{\substack{(a \colon j \to i) \\ \in \Cospan(i)}} T(j_{t+1}) \longar 0.
\]
Also the characteristic tilting module for our quasi-hereditary structure on $\NsQ$ is
\[
	T := \bigoplus_{t=1}^s \bigoplus_{i \in Q_0} T(i_t).
\]
\end{prp}

\begin{proof}
All composition factors of $T(i_t),\Delta(i_t)$ and $T(j_{t+1})$ have layer $\geq t$. 
Thus we can work in the full subcategory $\NsQ/ (E_{t-1})$, so we assume $t=1$ without loss of generality.
Apply (\ref{eq:std-seq2}) to the right $\NsQ$-module $P(i_s)_{\Lambda_0} \otimes_{\Lambda_0} \Lambda/\Lambda_+$.
By the identity (\ref{eq:PM3}) we get the exact sequence
\[
0 \longar \Lambda/\Lambda_+ \otimes_{\Lambda_0} {_{\Lambda_0}I}(i_s) \longar I(i_s) \longar \bigoplus_{\substack{(a \colon j \to i) \\ \in \Cospan(i)}} \iota_1 ({_{\Lambda/(e_1)}I}(j_s)) \longar 0.
\]
But ${_{\Lambda_0}I}(i_s) \cong {_{\Lambda_0}P}(i_1)$, and since we assume $t=1$ the kernel of the sequence is isomorphic to  $\Delta(i_t)$.
Moreover, $I(i_s) = T(i_1)$ and $\iota_1 ({_{\Lambda/ ( e_1)}I}(j_s)) = T(j_2)$ by construction, hence this is the sequence above.
By induction with the sequence (\ref{seq:char}) we see the modules $T(i_t)$ are in $\cF(\Delta)$.

All injective modules are in $\cF(\nabla)$ in general.
Consequently, each of the modules $T(i_{t+1})$ has a filtration of $\nabla$-modules, coming from the $\nabla$-filtration of ${_{N_{s-t}(Q)}I}(i_s)$ in $\NsQ/ (E_t) \bil \fdmod$. 
To summarise, all the modules $T(i_t)$ are pairwise distinct, indecomposable, and belong to $\cF(\Delta) \cap \cF(\nabla)$.
That shows $T$ is the characteristic tilting module of $\NsQ$.
\end{proof}

For an algebra $A$ and $M \in A \bil \modu$ we define respectively the modules generated and cogenerated  by $M$ as
\begin{align*}
\gen(M) &:= \{ N \in \NsQ \mid \; \exists \text{ an epimorphism } M_0 \twoheadrightarrow N \text{ with } M_0 \in \add(M) \},\\
\cogen(M) &:= \{ N \in \NsQ \mid \; \exists \text{ a monomorphism } N \rightarrowtail M_0 \text{ with } M_0 \in \add(M) \}.
\end{align*}
We summarise the most important properties of $\NsQ$ and the subcategories $\cF(\Delta)$ and $\cF(\nabla)$ in Proposition \ref{prp:f-delta}.
Keep in mind that we have established two different $\Z$-gradings on $\NsQ$, producing the graded modules $\Lambda$ and $\Gamma$. Hence any $\NsQ$-module can equivalently be considered as an ungraded $\Lambda$- or $\Gamma$-module.
\begin{prp}
\label{prp:f-delta}
The algebra $\NsQ$ has a quasi-hereditary structure, uniquely determined by the layer function $\tell$, that makes it simultaneously left strongly quasi-hereditary and right ultra strongly quasi-hereditary.
In particular the category $\cF(\Delta)$ is closed under taking submodules and $\cF(\nabla)$ is closed under taking factor modules.
Moreover the following are equivalent for an $\NsQ$-module $M$.
\begin{enumerate}[$(a)$]
\item 
$M \in \cN$.
\item
$M \in \cogen(I)$.
\item
$M \in \cF(\Delta)$.
\item
The underlying $\Lambda_0$-module of $M$ is projective.
\item
For the corresponding $(\Gamma_0,\Gamma_1)$-module $(M,\varphi)$, the map $\varphi$ is a monomorphism.
\end{enumerate}
Also the following conditions are equivalent:
\begin{enumerate}[$(a')$]
\addtocounter{enumi}{1}
\item
$M \in \gen(I)$.
\item 
$M \in \cF(\nabla)$.
\item
The underlying $\Gamma_0$-module of $M$ is injective.
\end{enumerate}

\end{prp}

\begin{proof}
Proposition \ref{prp:seqS1S2} already shows $\NsQ$ is left strongly quasi-hereditary.
By Corollary \ref{cor:equivalence} we see $\NsQ$ fulfils condition (US1). Moreover $\rad(\Delta(i_t)) =0$ if and only if $t=s$, but then $I(i_t) = T(i_1)$, and hence $\NsQ$ satisfies (US2). That shows $\NsQ$ is right ultra strongly quasi-hereditary.

Note that $\cN$ can be characterised as containing exactly the modules such that all summands of the socle have the form $S(i_s)$ for some $i \in Q_0$. 
Recall that the indecomposable summands of $I$ are $I(i_s)$ for $i \in Q_0$, thus $\cN = \cogen(I)$.
All the costandard modules have injective dimension at most 1 which, 
by \cite[Lemma 4.1*]{DRi3}, is equivalent to $\cF(\Delta)$ being closed under submodules.
By Proposition \ref{prp:char-tilt} we know $I(i_s) = T(i_1)$ is in $\cF(\Delta)$, thus $\cogen(I) \subset \cF(\Delta)$.
Now $\cogen(I)$ is closed under taking extensions because $I$ is injective. Since the standard modules are in $\cogen(I)$ that shows $\cF(\Delta) \subset \cogen(I)$.

Conditions $(c)$ and $(d)$ are equivalent by Corollary \ref{cor:equivalence},
and Condition $(e)$ is equivalent to $(a)$ by Proposition \ref{prp:N=monos}.

The equivalence of $(c')$ and $(d')$ is clear from the equivalence in Corollary \ref{cor:equivalence}.
From the characterisation of left strongly quasi-hereditary algebras in the appendix of \cite{Ri4}, $\cF(\nabla) = \gen(T)$. But the sequence (\ref{seq:inj}) shows that $\gen(I) = \gen(T)$, thus $(b')$ and $(c')$ are equivalent.
\end{proof}

\begin{remark}
\label{rem:ADR}
Let $A$ be an algebra with finite Loewy length as a left module, and let $s$ be such that $\rad(A)^s = 0$. Define
\[
M_A := \bigoplus_{t=1}^{s} A/ \rad(A)^t,
\]
and let $E$ be a basic $A$-module with $\add(E) = \add(M_A)$. 
The ADR-algebra is defined as $\cR_A := \End_A(E)^{\op}$. 
In \cite{Co} Conde proved that the ADR-algebras are right ultra strongly quasi-hereditary.
Indeed $\NsQ$ has many similarities with the ADR-algebras. In particular $\NsQ$ is isomorphic to $\cR_{kQ}$ as a quasi-hereditary algebra if $M_{kQ}$ is basic, which holds true if $Q$ has no sinks.

The Ringel-dual of ADR-algebras was calculated by Conde-Erdmann in \cite{CE}, which already covers the Ringel dual for many cases, where the Ringel dual of $\NsQ$ is $N_s(Q^{\op})^{\op}$.
The proof that this formula holds in general is to appear elsewhere \cite{Ei2}.


\end{remark}

\section{Richardson orbits from intermediate extensions}
\label{sec:RO from interm. ext.}

\subsection{Recollements from idempotents}
\label{subsec:idempotent recollement}
Let $B $ be a finite dimensional algebra and $e\in B $ be an idempotent element. We set $P=B e$ and $I=\kdual (eB )$ and $A=\End_{B }(P)^{\op}=eB e$. Left multiplication with $e$ gives a functor $e\colon B $-$\modu \to A$-$\modu$ which determines a recollement (see below). 
For every additive functor $F\colon \cA \to \cB$ we denote by $\Ker F$ the full subcategory of $\cA$ of objects $X$ with $F(X)=0$, and let $\Bild F$ denote the full subcategory of $\cB$ of objects $Y$ such that $Y\cong F(Z)$ in $\cB$ for some $Z$ in $\cA$, i.e. $\Bild F$ is the essential image of $F$.

We look at a diagram of six additive functors 
\[
\begin{tikzcd}
	B/(e) \bil \fdmod \arrow[rr,"i" description] 
	& & \arrow[ll,shift right=2ex,"q"'] \arrow[ll,shift left=2ex,"p"]  B \bil \fdmod \arrow[rr,"e" description] 
	& & \arrow[ll,shift right=2ex,"\ell"'] \arrow[ll,shift left=2ex,"r"] A \bil \fdmod
\end{tikzcd}
\]

defined by 
\[
\begin{aligned}
q & = B /(e) \otimes_{B } -,  &&& \ell &= P\otimes_A - , \\
i & = \text{ inclusion},   &&&  e &= \Hom_{B }(P,-),   \\
p & = \Hom_{B }(B / (e), -),  &&& r &= \Hom_{A} (eB,- ).
\end{aligned}
\] 
They fulfil: $(q,i,p), (\ell , e, r)$ are adjoint triples, $i,  \ell, r$ are fully faithful and $\Ker e =\Bild i$, which are the defining properties of a recollement of abelian categories. 

The associated TTF-triple in $\Gamma $-$\modu$ is given by 
\[ 
{\rm TTF}(e):= (\Ker q, \Ker e , \Ker p).
\] 
In particular, one has $\Ker e \cap \Ker p= \{0\}$.
Since $\ell $ is fully faithful, the unit $\eta\colon 1\to e\ell $ is an isomorphism. For $\eta^{-1}\colon e\ell \to 1$ there is an adjoint map $\ell \to r$. We define the \emph{intermediate extension functor} to be 
\[
c= \Bild (\ell \to r) \colon A\text{-} \modu \to B \text{-}\modu . 
\]

Now, we define for the projective $B$-module $P=Be$ and the injective $B$-module $I=\kdual (eB)$ the following full subcategories of $\Gamma \bil \fdmod $:
\[ 
\begin{aligned}
\gen_1(P) &:= \{ X \in B \bil \fdmod \mid \exists \: P_1\to P_0\to X\to 0 \text{ exact, } P_1, P_0\in \add (P) \}, \\
\cogen^1(I) &:= \{ X \in B \bil \fdmod \mid \exists \: 0\to X\to I_0\to I_1 \text{ exact, } I_1, I_0\in \add (I) \} .
\end{aligned}
\]

\begin{lemma} \label{KernelsOfFunctors}
\[
\begin{aligned}
 \Ker q &= \gen (P),  &&& \Bild \ell& =\gen_1 (P), \\
\Ker p &= \cogen (I), &&& \Bild r &=\cogen^1 (I).
\end{aligned}
\]
Therefore the essential image of the intermediate extension $c=\Bild (\ell \to r)$ has the description 
\[ 
\Bild c = \Ker q \cap \Ker p = \gen (P) \cap \cogen (I).
\]
\end{lemma}

\begin{proof} See \cite{PS}, Lemma 4.1.
\end{proof}


% ...


Recall that we say an $A$-module $M$ is rigid if $\ext^1_A(M,M)=0$.
\begin{lemma}
Let $B $ be a finite-dimensional algebra and $e\in B $ be an idempotent. 
We set $A=eB e$. If $M$ is a rigid $A$-module, then $c(M)$ is a rigid $B $-module. 
\end{lemma}

\begin{proof} 
We have $\ell (A) = B e $, so $\ell $ maps projective modules to projectives. 
Choose a projective cover $p_M\colon P_M \to M$ in $A \bil \fdmod$ and apply $\ell$ to it to get $\ell (p_M) \colon \ell (P_M ) \to\ell ( M)$. 
This is an epimorphism because $\ell $ is right exact. We compose it with the canonical epimorphism $\ell (M) \to c(M)$ and obtain a short exact sequence of $B $-modules
\[ 
0\to Y\to \ell (P_M) \to c(M) \to 0.
\]
Now, we apply $\Hom_{B}(- ,c(M))$ to it and obtain an exact sequence 
\[ 
\begin{aligned}
0 &\to \Hom_{B }(c(M), c(M))  \to \Hom_{B} (\ell (P_M) , c(M)) \xrightarrow{g} \Hom_{B} (Y, c(M)) \\
& \to \ext^1_{B}(c(M), c(M)) \to 0,
\end{aligned}
\]
using that $\ext^1_{B }(\ell (P_M), c(M))=0$ since $\ell (P_M)$ is projective.  
Since $ec\cong \idt $ and $\ell $ is a left adjoint to $e$, the first two vector spaces can be identified as 
\[ 
\Hom_{B }(c(M), c(M)) \cong \Hom_A (M,M), \quad \Hom_{B} (\ell (P_M), c(M) ) \cong \Hom_{A} (P_M, M).
\]
We also know that the functor $e$ gives a map $\Hom_{B } (Y, c(M)) \to \Hom_A( eY, M)$, and this coincides with the monomorphism $\Hom_B(Y,c(M)) \to \Hom_B(Y,r(M)) \cong \Hom_A(eY,M)$ induced by the inclusion $c(M) \to r(M)$. 
Since $e$ is exact, we get an exact sequence $0\to eY \to P_M \to M \to 0$. We apply the functor $\Hom_A(-, M)$ and get an exact sequence 
\[ 
0\to \Hom_{A }(M, M) \to \Hom_{A} (P_M , M) \xrightarrow{f} \Hom_{A} (eY, M) \to \ext^1_{A}(M, M) \to 0.
\]
We observe $\dim \Bild f = \dim \Bild g$. An easy comparison of dimension gives us 
\[ 
\begin{aligned}
\dim \ext^1_{B} (c(M), c(M)) & = \dim \Hom_{B}(Y, c(M)) - \dim \Bild g \\
& \leq \dim \Hom_A( eY, M) - \dim \Bild f = \dim \ext^1_A(M, M) .
\end{aligned}
\] 
So, if $M$ is rigid, then $c(M)$ is also rigid.
\end{proof}

\subsubsection{The associated projective maps. }
\label{subsub:asso-proj}

Let $k$ be an algebraically closed field. Let $A$ be a finite-dimensional basic $k$-algebra given by the quotient of a path algebra $kQ$ by an admissible ideal $I$. 
We say an $A$-module $M$ is rigid if $\ext^1_A(M,M)=0$.
Let $d \in \N_0^{Q_0}$ and denote by $\rep_{d} (A) \subset \rep_{d}$ the subset of points corresponding to modules $M$ with $IM=0$, note that this subset is $\GL_d$-invariant.
This closed subset comes with a natural scheme structure $\repsch_d(A)$, with $A$-modules of dimension vector $d$ corresponding to $k$-valued points of $\repsch_d(A)$. 
We can see $\rep_d(A)$ as the variety given by the reduction of this scheme, it inherits the action of $\GL_d$ and there is bijection between isomorphism classes of $A$-modules of dimension $d$ and $\GL_{d}$-orbits in $\rep_{d}(A)$. 
Note that $\rep_d(A)$ is not necessarily irreducible.

We recall the following facts from the literature. Part (1) is a corollary of Voigts lemma, cf. \cite{Ga}.
The second fact is stated and proven in \cite[Propostion 3.7]{Ge}, but it comes from earlier folklore.

\begin{lemma}
\label{lem:schemes}
\begin{itemize}
\item [(1)]
Let $M$ be a smooth $k$-valued point in $\repsch_d(A)$. Then the $\GL_d$-orbit of $M$ in $\rep_d(M)$ is open if and only if $M$ is rigid.
\item[(2)]
Let $M$ be a $k$-valued point of $\repsch_d(A)$. If $\ext_A^2(M,M) = 0$, then $M$ is a smooth point. 
\end{itemize}
\end{lemma}

Let $B$ be a finite-dimensional basic $k$-algebra. 
Let $e \in B $ be a sum of primitive orthogonal idempotent elements  and $A=e B e$. 
We can assume $B$, and hence $A$, is given by a quiver with relations, where the vertices correspond to primitive idempotents.
For a $B$-module $X$ we write $\Dim X= (f,d)= :\dd $, where $d=\Dim eX$ as a dimension vector for $A$ and $f=\Dim (1-e) X$ as a dimension vector for $(1-e)B(1-e)$. 
We look at the algebraic map 
\[ 
e \colon \rep_{\dd} (B) \to \rep_d (A), \quad X \mapsto eX.
\] 
We denote $\rep_{\dd} (B )^{\st}:=\{X\in \rep_{\dd}(B )\mid X\to re (X) \text{ is a monomorphism }\}\subseteq \rep_{\dd} (B)$ and call this the subset of stable points.
Equivalently, we can characterise the stable points as all $X \in \rep_{\dd}(B)$ such that $X \in \cogen I$.
By \cite[Section 7]{CBSa}, this is an open subset. 
We also recall: If $X$ is stable, then the natural monomorphism $ce(X)\subset re(X)$ factors over the other natural monomorphism $X \to re(X)$ implying that we have a natural monomorphism $ce(X) \to X$.  
The geometric quotient $\rep_{\dd} (B )^{\st} /\GL_f$ exists, as well as a projective map 
\[ 
\pi \colon \rep_{\dd} (B )^{\st}/\GL_f \to \rep_d(A), \quad [X] \mapsto eX.
\]

\begin{prp}
\label{prp:semi-cont}
Let $n$ denote the number of primitive idempotents in $B$.
Consider the dimension vectors of $r$ and $c$. We denote by $\cR_{\dd}$ (resp. $\cC_{\dd}$) the subset of points $M$ in $\rep_d (A)$ with $\Dim r (M) = \dd$ (resp. $\Dim c(M) = \dd $).
\begin{itemize}
\item[$(1)$]
The map 
$\rep_d(A) \to \N_0^n$, defined by $M \to \Dim r(M)$, is upper semi-continuous in each coordinate.\\ 
The map $\rep_d(A) \to \N_0^n$, defined by $M \to \Dim c(M)$, is lower semi-continuous in each coordinate.

\item[$(2)$] The functors $r$ and $c$ induce isomorphisms of algebraic varieties 
$r_{\dd}\colon \cR_{\dd} \to \pi_{\dd}^{-1} (\cR_{\dd} )$  and  $c_{\dd}\colon \cC_{\dd} \to \pi_{\dd}^{-1} (\cC_{\dd} )$, with the inverse given by restricting $\pi_{\dd}$

\end{itemize}
\end{prp}

\begin{remark}
For completeness we point out that similar results as for the functor $r$ hold for the functor $\ell $. 
In that case the analogue map $\ell_{\dd} $ will take values in $ \rep_{\dd} (B )^{\rm cost}/\GL_f $, where $\rep_{\dd} (B)^{\rm cost} =\{X\in \rep_{\dd}(B )\mid \ell e(X)\to X \text{ is an epimorphism }\}\subseteq \rep_{\dd} (B)$ and we call this the subset of costable points. 
\end{remark}

Let $Q_0$ denote the set of primitive idempotents of $A$.
For a dimension vector $d \in \N_0^{Q_0}$ we let $k^d$ denote the $Q_0$-graded vector space $\bigoplus_{i \in Q_0} k^{d_i}$. 
If $d,d' \in \N^{Q_0}$ are dimension vectors, we let $\rm M_{d \times d'}(k)$ denote the affine variety of homomorphisms of $Q_0$-graded vector spaces $k^d \to k^{d'}$.

\begin{proof}
Let us first prove $(1)$.
Let $e_i$ be a primitive idempotent of $B$, not necessarily a summand of $e$.
Then $eBe_i$ is a left $A$-module. It has a free resolution
\[
	A^p \overset{\chi}{\to} A^q \to eBe_i \to 0.
\]
Let $d \in \N_0^{Q_0}$ and $M \in \rep_d(A)$, recall that $r(M) = \Hom_A(eB,M)$.
Moreover the $i$-th coefficient of the dimension vector of $r(M)$ is given by $\dim_k e_i r(M) = \dim_k \Hom_A(eBe_i,M)$.
Apply $\Hom_A(-,M)$ to the free resolution of $eBe_i$ to obtain the exact sequence
\[
\label{eq:chi-res}
\tag{$\chi^*$}
\begin{tikzcd}
0 \rar & \Hom_A(eBe_i,M) \rar & \Hom_A(A^q,M) \rar["f_M"] & \Hom_A(A^p,M).
\end{tikzcd}
\]
Clearly $\dim_k \Hom_A(eBe_i,M) = \dim \ker f_M$, and $\Hom_A(A^q,M) = (k^d)^q = k^{qd} $. 
From a construction in \cite[Section 3.3]{Z02} we know that the map $\chi$ induces a morphism of varieties $f \colon \rep_d(A) \to {\rm M}_{pd \times qd}(k)$.
Now observe that the map $f_M$ is isomorphic to the map $f(M)$.
But the dimension of the kernel of a matrix over $k$ is upper semi-continuous, thus $M \mapsto \dim_k e_i r(M)$ is upper semi-continuous.\\
To see that the map  $M \mapsto \Dim c(M)$ is lower semi-continuous in each coordinate we combine lemmas 6.3 and 7.2 in \cite{CBSa}. The first lemma implies that the image of the collapsing map is always closed.
Then Lemma 7.2 states that the image of the collapsing map $\pi_{\dd}$ can be characterised by $\{ M \in \rep_d(A) \mid \Dim c(M) \leq \dd \}$, where the ordering of dimension vectors is given by pointwise comparison. This implies the lower semi-continuity. \\

Now we look at (2). Let $e_1, \ldots , e_n$ be a complete set of primitive idempotents for $B$ and we assume without loss of generality that $e=e_1 +e_2 + \cdots +e_m$ for an $m \leq n$. 
We have a regular map $\rep_{\dd} (B)^{\rm st} \to \rep_d(A) \times \prod_{m+1}^n{\rm Gr}\binom{\Hom_k(eBe_i, k^d)}{f_i} $ mapping $M\mapsto (eM, e_i(M) \subset \Hom_k(eBe_i, k^d)) $. 
The image is a closed subset isomorphic to $\rep_{\dd}(B)^{\st} /\GL_f $. Therefore, it is enough to check that 
the maps
\begin{align*} 
r_{i} &\colon \cR_{\dd } \to {\rm Gr}_k \binom{\Hom_k(eBe_i,M)} {f_i}, \; M \mapsto \left( \Hom_A(eBe_i,M) \subset \Hom_k(eBe_i,  k^d)\right); \\
c_{i} &\colon \cC_{\dd} \to {\rm Gr}_k \binom{\Hom_k(eBe_i,M)} {f_i}, \; M \mapsto \left( e_ic(M) \subset \Hom_k(eBe_i, k^d)\right) ;
\end{align*}
are regular for $i=m+1,\dots,n$.
 
A $k$-valued point $M$ of $\rep_d(A)$ gives the map $f_M$ in the exact sequence (\ref{eq:chi-res}).
As established in the proof of $(1)$ the map $\rep_d(A) \to {\rm M}_{pd\times qd}(k) $ mapping $M$ to $f_M$ is a regular map. 
By imposing rank conditions this induces a regular map 
$\cR_{\dd} \to \{ f \in {\rm M}_{pd\times qd}(k) \mid \dim \Ker (f) = f_i \}$. Now, we can compose this map with the  principal $\GL_{f_i}$-bundle 
$\{ f \in {\rm M}_{pd\times qd}(k) \mid \dim \Ker (f) = f_i \} \to {\rm Gr}_k \binom{k^{qd} } {f_i}$ . But this was not yet the Grassmannian we wanted to end up in. We have a commutative diagram 
\[
\begin{tikzcd}
0 \ar[r] & \Hom_A(eBe_i,M) \ar[r] \ar[d] & \Hom_A(A^q,M) \ar[r,"f_M"] \ar[d] & \Hom_A(A^p,M)\ar[d] \\
0 \ar[r] & \Hom_k(eBe_i,M) \ar[r] & \Hom_k(A^q,M) \ar[r]& \Hom_k(A^p,M),
\end{tikzcd}
\]
where the arrows pointing down are the natural inclusions. This induces closed immersions of Grassmannians 
\[  
\begin{tikzcd}
{\rm Gr}_k \binom{ k^{qd}} {f_i} \ar[r]&  {\rm Gr}_k \binom{\Hom_k(A^q,k^d)} {f_i} & {\rm Gr}_k \binom{\Hom_k(eBe_i,k^d)} {f_i}\ar[l] .
\end{tikzcd}
\]
By composition we get a regular map $\cR_{\dd} \to {\rm Gr}_k \binom{\Hom_k(A^q,k^d)} {f_i}$. In fact, the image lies in the closed subvariety ${\rm Gr}_k \binom{\Hom_k(eBe_i,k^d)} {f_i}$ and therefore $r_{i}$ is a regular map.
 
 For the definition of $c_i$ we consider $ \rep_d(A) \to \Hom_k (e_i Be \otimes_k k^d,  \Hom_k(eBe_i, k^d) )$ 
 mapping $M$ to the composition $g_M\colon e_i Be \otimes_k M \to e_iBe\otimes_A M \to \Hom_A(eBe_i, M) \subset  \Hom_k(eBe_i, k^d)$, where the map in the middle comes from the natural transformation $\ell \to r$.  
 Again, by considering $R$-valued points, it is easy to see that this is a regular map. By imposing rank conditions this induces a morphism of varieties 
 $\cC_{\dd} \to \{g \in \Hom_k (e_i Be \otimes_k k^d,  \Hom_k(eBe_i, k^d) ) \mid \dim \Bild g = f_i \}$. 
We compose this with the principal $\GL_{f_i}$-bundle 
\[
\{g \in \Hom_k (e_i Be \otimes_k k^d,  \Hom_k(eBe_i, k^d) ) \mid \dim \Bild g = f_i \} \to {\rm Gr }_k \binom{ \Hom_k (eBe_i , k^d)}{ f_i}.
\]
By deinition of the functor $c$, this gives us the map $c_i$ as a composition of regular maps. This shows $c_i$ is regular. 

The last claim is that $r_{\dd}$ and $c_{\dd}$ are isomorphisms. 
This follows directly from $er={\rm id} = ec$, which implies the inverses are induced by the restrictions of the map $e \colon \rep_{\dd}(B) \to \rep_d(A) $ to the fibres $\pi_{\dd}^{-1}(\cR_{\dd})$ and $\pi_{\dd}^{-1}(\cC_{\dd})$ respectively.  
\end{proof}

For every point $M\in \rep_d(A)$ we denote by $\pi^{-1} (M)$ the underlying reduced variety of the scheme $\pi^{-1} (M)$. 
If $N$ is a $B$-module with $eN=0$ then we can see $N$ naturally as a $B/(e)$-module and denote by $\Gr_{B /(e)} \binom{N}{f}$ the quiver Grassmannian with the underlying reduced scheme structure. 

\begin{lemma} \label{fibre lemma}
There is an $\Aut_A(M)$-equivariant isomorphism of projective varieties 
\[ 
\pi^{-1}(M) \to \Gr_{B /(e)} \binom{r(M)/c(M)}{\dd- \Dim c(M)} .
\]
In particular, we have $\Bild \pi \subset \{ M \in \rep_d(A) \mid \Dim c(M) \leq \dd \leq \Dim r(M)\}$. 
\end{lemma}

Before we can prove the lemma, we need the following lemma: 
\begin{lemma} Let $d>f>e\geq 1$ be natural numbers. We denote by $\fl (e,f,d)$ the flag variety of flags $U_1 \subset U_2 \subset k^d$ with $\dim U_1=e, \dim U_2=f$ and by 
\[ 
p\colon \fl (e,f,d) \to \Gr (e,d),
\]
the map given by $p(U_1\subset U_2 \subset k^d)= (U_1 \subset k^d)$. This is a regular map between projective varieties. The map
\[ 
\phi \colon p^{-1}(U_1) \to \Gr (d-e, f-e),  
\]
defined by $\phi (U_2)= U_2/U_1$, is an isomorphism of varieties. 
\end{lemma}



\begin{proof}[Proof of Lemma \ref{fibre lemma}]
A $k$-valued point in $\pi^{-1}(M)$ is given by a $B$-module $X$ fitting in a flag of submodules 
\[ 
c(M)=ce(X) \subset X \subset re(X) = r(M),
\]
such that $\Dim X = \dd $. Observe here that a point in $\rep_{\dd} (B )^{\rm st} $ is a point $X$ where these two morphisms are monomorphisms. 
Moreover $ce(X) \subset re(X)$ by definition. 
Now, passing to the geometric quotient under the $\GL_f$ operation on $X$, the two monomorphisms can be assumed to be set-theoretic inclusions. This means that $\pi^{-1}(M) $ is naturally a closed subscheme of a fibre $p^{-1}(U_2)$ as discussed in the auxiliary lemma. This implies that the map 
\[ 
\pi^{-1}(M) \to \Gr_{B /(e)} \binom{r(M)/c(M)}{\dd- \Dim c(M)} ,
  \]
sending $X$ with $c(M) \subset X \subset r(M)$ to $X/c(M)$, is an isomorphism of varieties. 
\end{proof}


\subsection{Application of the recollement to $\NsQ$}
\label{subsec:nsq recollement}

We apply the recollement above to $B = \NsQ$.
We take $e\in \NsQ$ to be the idempotent corresponding to the sum over the primitive idempotents in the $s$-th layer, i.e. $e = \sum_{i \in Q_0} e(i_s) $. 
Let $J\subset kQ$ be the two-sided ideal generated by $Q_1$. Then $kQ/J\cong k^{Q_0}$ is semi-simple and $kQ /J^s$ is finite-dimensional. 
It is well-known that one has $J=\rad (kQ)$, the Jacobson radical of $kQ$, if and only if $Q$ contains no oriented cycles. 

\begin{lemma} \label{eNsQe}
There is a canonical isomorphism of finite-dimensional $k$-algebras 
\[
e\NsQ e\to kQ/J^s.
\]
\end{lemma}

\begin{proof}
We build a surjective ring homomorphism $\phi \colon kQ \to e \NsQ e$.
Let\\ $\alpha = a^{(l)} \cdots a^{(1)} \colon i \to j$ be a path of length $l< s$ in $kQ$,
we write
\[
	a(\alpha) := a_{s-l+1}^{(l)} \cdots a_{s}^{(1)}.
\]
Let $b(\alpha) := b(j_{s-1}) \cdots b(j_{s-l})$. Then we define $\phi(\alpha) := a(\alpha) b(\alpha)$.
If $\alpha$ is a path of length $\geq s$ we define $\phi(\alpha) := 0$.

To show that $\phi$ is a ring homomorphism, we check this on basis elements. 
Let $\alpha,\alpha'$ be paths of length $l$ and $l'$ respectively. 
If either $l$ or $l'$ is $\geq s$, then clearly $\phi(\alpha \alpha') = \phi(\alpha) \phi(\alpha') = 0$.
If the source of $\alpha$ is not equal to the target of $\alpha'$, then the source of $\phi(\alpha)$ is not equal to the target of $\phi(\alpha')$, thus $\phi(\alpha)\phi(\alpha') = 0 = \phi(0) = \phi(\alpha \alpha')$.
If $\alpha \alpha' \neq 0$ and $l + l' \geq s$, then the path $\phi(\alpha)\phi(\alpha')$ is contained in the ideal $\mathfrak{I}$, thus $\phi(\alpha \alpha') = 0 = \phi(\alpha) \phi(\alpha')$.
The remaining case is when $\alpha \alpha' \neq 0$ and $l + l' < s$. But then $\phi(\alpha) \phi(\alpha') = \phi(\alpha \alpha')$ up to the relations in $\NsQ$. Thus $\phi$ is a ring-homomorphism with kernel $J^s$.
 
Now $e\NsQ e$ has as standard basis elements all paths of the form $\beta \alpha$, where $\beta = b(j_{s-1}) \cdots b(j_t)$ and 
$\alpha = a_{t+1}^{(t+1)} \cdots a_s^{(s)}$.
But those are obtained as $\phi(a^{(t+1)} \cdots a^{(s)}) = \beta \alpha$, thus $\phi$ is surjective.
Then $e \NsQ e \cong kQ/J^s$ by the first isomorphism theorem. 
\end{proof}

We identify $e \NsQ e$ with $kQ/J^s$ via the isomorphism in Lemma \ref{eNsQe}, then the recollement corresponding to $e\in \NsQ$ takes the form 
\[
\begin{tikzcd}
	\NsQ/(e) \bil \fdmod \arrow[rr,"i" description] 
	& & \arrow[ll,shift right=2ex,"q"'] \arrow[ll,shift left=2ex,"p"]  \NsQ \bil \fdmod \arrow[rr,"e" description] 
	& & \arrow[ll,shift right=2ex,"\ell"'] \arrow[ll,shift left=2ex,"r"] kQ/J^s \bil \fdmod.
\end{tikzcd}
\]

By Proposition \ref{prp:f-delta} and Lemma \ref{KernelsOfFunctors} we know that for the $\Delta $-filtered modules of the quasi-hereditary structure from the previous section the following holds:
\[ 
\begin{aligned}
\Bild  c = \cogen (I) \cap \gen (P) & \subset \cogen (I) =  \cF (\Delta ), \\
\Bild  r = \cogen^1(I) & \subset \cogen (I) =  \cF (\Delta ) .
\end{aligned}
\]
Therefore, we can restrict this to functors (which we still denote by the same letters) 
\[
\begin{tikzcd}
	\NsQ/(e) \bil \fdmod 
	& & \arrow[ll,shift right=2ex,"q"']   \cF(\Delta) \arrow[rr,"e" description] 
	& & \arrow[ll,shift right=2ex,"c"'] \arrow[ll,shift left=2ex,"r"] kQ/J^s \bil \fdmod.
\end{tikzcd}
\]
The following properties of the restricted functors are straightforward and follow from the properties of the recollement. 
\begin{itemize} 
\item[(1)] $e$ is faithful and exact; 
\item[(2)] $(c , e )$ and $(e, r)$ are adjoint pairs and $c, r$ are both fully faithful;
\item[(3)] $r$ maps injectives to injectives and $c$ maps projectives to projectives. 
\end{itemize}




Let $M$ be a $kQ/J^s$-module, we can determine $c(M)$ and $r(M)$ explicitly.
Consider the socle filtration of $M$.
It is defined recursively by $\soc^0(M) = 0$, and $\soc^{t+1}(M)$ is the submodule of $M$ determined by
\[
	\soc^{t+1}(M)/\soc^{t}(M) = \soc (M /\soc^{t}(M)).
\]
Equivalently we can define $\soc^t(M)$ as the maximal submodule of $M$ annihilated by $J^{t}$.
In particular we have $\soc^t(M) = M$ for all $t \geq s-1$.
Recall that we can consider $r(M)$ and $c(M)$ as objects of $\mon_s(Q)$ via the functor $\Phi^*$ from Section \ref{subsub:N and tensor algebra}.
We write $N_t := \Phi^*(N)_t$ for the $t$-th $kQ$-module of $\Phi^*(N)$.
\begin{lemma}
\label{lem:dim c}
Let $M$ be a $kQ/J^s$-module. Then
\begin{itemize}
\item[$(1)$]
$r(M)_t \cong \soc^t(M)$.
\item[$(2)$]
$c(M)_t \cong J^{s-t} M$.
\end{itemize}
%
%
\end{lemma}

\begin{proof}
Let $e_t := \sum_{i \in Q_0} e(i_t)$ for $t=1,\dots,s$, in particular $e_s = e$.
By Lemma \ref{eNsQe} we know $e_t \NsQ e_t \cong kQ/J^t$ as $kQ$-modules, and we obtain the $kQ$-module $r(M)_t$ as $e_t r(M)$. 
Observe that $e \NsQ e_t \cong e\NsQ e / J^t$ as a left $kQ$-module, we obtain identities
\begin{align*}
	e_t r(M) = \Hom_{kQ}(e\NsQ e_t, M) \cong \Hom_{kQ}(kQ/J^t,M) \cong \soc^t(M).
\end{align*} 
Thus $r(M)_t \cong \soc^t(M)$ as claimed.

By construction $c(M)$ is the maximal submodule of $r(M)$ generated by the projective $\NsQ$-module $\NsQ e$, i.e. $c(M) = \NsQ e \: r(M)$.
Moreover, since $c(M) \in \cN$, we know $e_t \: c(M) \cong e_t \NsQ  e \: c(M)$ as a $kQ$-module.
Now $e r(M) \cong M$ and by the relations on $\NsQ$ we have $e \NsQ e_t \NsQ e = J^{s-t} \subset e \NsQ e$.
Combining this we get the following identities of $kQ$-modules:
\[
c(M)_t = e_t c(M) = e_t \NsQ e \: r(M) \cong e \NsQ e_t \NsQ e \: e r(M) \cong J^{s-t} M.
\]



\end{proof}


\subsubsection{Richardson orbits from intermediate extensions}
\label{RO-main}

Take a dimension filtration $\dd = (\dd^{(1)}, \ldots , \dd^{(s)})$ of $d \in \N_0^{Q_0}$, $\dd$ can be seen as a dimension vector of $\NsQ$ in a canonical way, with $d_i^{(t)}$ corresponding to the vertex $i_t$ of $\Qs$.
We set $f:=(\dd^{(1)}, \ldots , \dd^{(s-1)})$, so $\dd =(f,d)$, then by \cite{CBSa} we have a free operation of $\GL_f = \prod_{t=1}^{s-1} \GL_{\dd^{(t)}} $ on $\repst$.
By Proposition \ref{prp:f-delta} we know that an $\NsQ$-module is stable w.r.t. $e$ if and only if it is in $\cN = \cogen (I)$, where $I= \kdual (e \NsQ )$. 
Thus $\repst$ is the subset of $\rep_\dd(\NsQ)$ of $\dd$-dimensional $\NsQ$-modules which are in the subcategory $\cN$. 
Taking the quotient module $\GL_f$ amounts to taking the flag of submodules determined by an object in the monomorphism category, hence we can identify the geometric quotient $\repst/\GL_f$ with $\afb$. 
The multiplication with the idempotent e induces a map of algebraic varieties 
\[ 
\varphi \colon \repst \to \rep_d (kQ/J^s), \quad M\mapsto eM.
\] 
Then the following diagram commutes 
\[ 
\begin{tikzcd}
\afb \ar[rr,"\pi_{\dd}"] & & \rep_d(kQ/J^s) \\
\repst \ar[u,"\pi_{\cN}"] \ar{urr}[swap]{\varphi} & &
\end{tikzcd}
\]
In particular $\repst$ is irreducible and smooth. 

The map $\pi_{\cN} \colon \repst \to \afb$ induces a bijection between dense $\GL_d$-orbits in $\afb$ and dense $\GL_{\dd}$-orbits in $\repst$. 
Recall from the previous subsection that the category $\cN$ is the category of $\Delta$-filtered modules for a quasi-hereditary structure.
Therefore, we obtain the following. 

\begin{thm} \label{bijection}
We consider the algebra $\NsQ$ with the quasi-hereditary structure from subsection \ref{subsec:QH}. The following are equivalent:
\begin{itemize}
\item[(1)] There is a rigid $\Delta$-filtered $\NsQ$-module of dimension vector $\dd$.
\item[(2)] There is a Richardson orbit for $(Q,\dd)$. 
\end{itemize}
\end{thm}

\begin{proof}
Let $M$ be a point in $\rep_{\dd}(\cN)$, then $M$ is a $\Delta$-filtered $\NsQ$-module and hence has projective dimension at most 1. 
But then Lemma \ref{lem:schemes} implies that the corresponding $k$-valued point of $\repsch_{\dd}(\NsQ)$ is smooth.

Let $M$ be a rigid $\Delta $-filtered module of dimension $\dd$, such $M$ is unique up to isomorphism.
We can consider M as a point in $\rep_{\dd}(\cN)$, and $\GL_{\dd} \cdot M \subset \rep_d(\cN)$ is a dense orbit by Lemma \ref{lem:schemes}.
By our discussion above the $\GL_d$-orbit of $\pi_{\cN}(M) \in \afb$ is dense, and by Theorem \ref{thm:5equ} that means $(Q,\dd)$ has a Richardson orbit.

Now assume $\afb$ has a dense $\GL_d$-orbit $\cO$. Then $\pi_{\cN}^{-1}(\cO) \subset \rep_{\dd}$ is a dense orbit of $\rep_{\dd}(\cN)$. 
Since all points $M$ of this orbit are smooth points in $\repsch_{\dd}(\NsQ)$, Lemma \ref{lem:schemes} implies $M$ is rigid.
\end{proof}

By \cite[Lemma 7.2]{CBSa} we have for the intermediate extension $c$ associated to the idempotent $e$  
\[ 
\Bild \pi_{\dd} = \Bild \varphi \subset \{ Y\in \rep_d(kQ/J^s) \mid \Dim c(Y) \leq \dd\}. 
\] 
We also observe that for $M \in \rep_d(kQ/J^s)$, one can see $c(M)$ naturally as a point in $\repst$ for $\dd = \Dim c(M)$, which gives a point in $\afb$ via the quotient map.

\existRO*

\begin{proof} 
If $\dd = \Dim c(M)$ we know by Prop \ref{prp:semi-cont} (1) that $M\mapsto \Dim c(M)$ is lower semi-continuous and thus 
\[
\cC_{\dd} =\{ N \in \rep_d(A) \mid \Dim c(N) = \dd \} \subset  \{ N \in \rep_d (A) \mid \Dim c(N) \leq \dd \} = \Bild \pi_{\dd}
\]
is an open and non-empty subset of an irreducible variety.
By loc. cit. (2) we also get that the restriction of $\pi_{\dd}$ to 
$\pi_{\dd}^{-1} (\cC_{\dd} ) \to \cC_{\dd}$ is an isomorphism. 
In particular, $\pi_{\dd}$ is birational. \\
Now assume $\dd = \Dim r(M)$. By Prop \ref{prp:semi-cont} (1) the map $N\mapsto \Dim r(N)$ is upper semi-continous, so $\cR_{\dd} = \{N \in \rep_d(A) \mid \Dim r(N) = \dd \} \subset \{N \in \rep_d(A) \mid \Dim r(N) \geq \dd \}$ is an open non-empty subset.
Furthermore, we know by Proposition \ref{prp:semi-cont} (2) that $\cR_{\dd} \subset \Bild \pi_{\dd} \subset \{N \in \rep_d(A) \mid \Dim r(N) \geq \dd \}$, and the second inclusion is one of closed subsets in $\rep_d(A)$. Therefore $\cR_{\dd}$ is an open subset of $\Bild \pi_{\dd}$. 
Moreover Proposition \ref{prp:semi-cont} (2) states that the restriction of $\pi_{\dd}$ to $\pi_{\dd}^{-1} (\cR_{\dd}) \to \cR_{\dd}$ is an isomorphism, therefore $\pi_{\dd}$ is birational. 
\end{proof}

\begin{cor} 
The functor $r$ maps rigid modules to rigid modules. 
\end{cor}

\begin{cor}
If $kQ/J^s$ is representation-finite, then 
$(Q, \dd) $ has a Richardson orbit for every dimension $\dd $ of an intermediate extension $c(X)$, or of $r(X)$, with $X$ a $kQ/J^s$-module.  
\end{cor}

\begin{rem}
If we take $M\in \rep_d(kQ/J^s)$ and $\dd := \Dim c(M)$ and assume $\Bild \pi_{\dd} = \overline{\cO}_N$, then it can happen that $N$ and $M$ are not isomorphic. 
For example, this is the case in the example below. 
\end{rem}

\begin{example} \label{running} Take $Q= (\xymatrix{1 & \ar[r] 2 \ar[l] & 3})$. Then the quiver for $B = \NtwoQ$ is 
\[ 
\xymatrix{
1_2   &  2_2 \ar[dl] \ar[dr] & 3_2  \\
1_1 \ar[u] & 2_1 \ar[u] \ar@{.}[l] \ar@{.}[r]& 3_1. \!\! \ar[u] 
}
\]
We consider the idempotent recollement with respect to  $e=e(1_2) + e(2_2) + e(3_2)$, clearly $eBe=kQ$. We look at the $kQ$-module $M = I(2) \oplus P(2)$. 
Then we have $r(I(2))=I(2_2)$ since $r$ maps injectives to injectives, 
and $c(I(2))=S(2_2)$ because $c$ maps simples to simples. 
We have $r(P(2))=P(2_2) = c(P(2))$ because there is no other indecomposable module $Y$ with $eY=P(2)$. 
Then the module $X:=r(M) =I(2_2) \oplus P(2_2)$ is rigid, because $I(2_2)= P(2_1)$ is projective-injective. But $ce(X)= c(M)=S(2_2) \oplus P(2_2)$ is not rigid, because we have a non-split short exact sequence 
\[
0\to P(2_2) \to I(1_2) \oplus I(3_2) \to S(2_2) \to 0. 
\]
\end{example}

Furthermore, we have the following corollary from the proof of Theorem \ref{existRO}. The second statement is the main result of \cite{BHT}, but we provide a different proof. 

\begin{cor} \label{irrcompOfNilp} Let $\cC_{\dd} , \cR_{\dd} \subset \rep_d(kQ/J^s)$ be two strata defined as before through the functors $r$ and $c$ induced by the highest layer idempotent $e\in \NsQ$. 
Let $\leq $ denote the partial order on dimension vectors given by pointwise comparison in every coordinate.
\begin{itemize}
\item[$(1)$] Every non-empty $\cC_{\dd}$ and every non-empty $\cR_{\dd}$ is smooth, irreducible and of dimension 
\[ 
\dim \RF = d \cdot d - \langle \dd , \dd \rangle_{\NsQ},
\] 
and its closure is $\overline{\cC}_{\dd} = \Bild \pi_{\dd}$ and $\overline{\cR}_{\dd} =  \Bild \pi_{\dd}$. 
\item[$(2)$] For every non-empty $\cC_{\dd}$ one has $\overline{\cC}_{\dd} = \bigcup_{\dd' \leq \dd} \cC_{\dd'}$. The irreducible components of $\rep_d(kQ/J^s)$ are those $\overline{\cC}_{\dd}$ for which
$\dd$ is maximal with respect to the finite set 
$\{\dd \mid \exists \: M\in \rep_d(kQ/J^s) \text{ with } \Dim c(M)= \dd \}$. 
\end{itemize}
\end{cor}

\begin{proof} Statement (1) follows directly the proof of Theorem \ref{existRO}. \\
For (2), the first claim follows from $\Bild \pi_{\dd} = \{ N \in \rep_d(kQ/J^s) \mid \Dim c(N) \leq \dd \}$ cf. \cite[Lemma 7.2]{CBS}. The second claim is a direct consequence of it.
 \end{proof}
 

\subsection{The relaxed Richardson property}

Let $Q$ be as before, let $d$ be a dimension vector for $Q$, and let $\dd$ be a dimension filtration of $d$ of length $s$.
As we still consider the collapsing map $\pi_{\dd} \colon \afb \to \rep_d$.

\begin{defn} \label{DfnrelaxRP}
We say the pair $(Q,\dd)$ has the \emph{relaxed Richardson property} if there is an open subset $U \subset \Bild \pi_{\dd}$ such that for each point $M \in U$, the automorphism group $\aut_{kQ/J^s}(M)$ operates with a dense orbit on $\pi_{\dd}^{-1}(M)$.
\end{defn}

\begin{rem}
This is a weaker condition than having a Richardson orbit. 
By Theorem \ref{thm:5equ}, $(Q,\dd)$ has a Richardson orbit if and only if there is an open subset $U$ as in the definition which has a dense $\GL_d$ orbit.
In Section \ref{subsec:Kronecker} is a case where the relaxed Richardson property holds, but there is no Richardson orbit.
Section \ref{subsec:D4} we give an example where the relaxed Richardson property fails.
\end{rem}



We do however have a large class of examples where the relaxed Richardson property holds.

\relaxed*

\begin{proof}
Take the proof from Theorem \ref{existRO}. It also applies to this situation. 

\end{proof}


\subsection{The Grassmannian case $s=2$}

Now we study the case $s=2$. Let $e=\sum_{i\in Q_0} e(i_2) \in \NtwoQ$.  
We observe that $\NtwoQ/(e)$ is semi-simple. This already leads us to: 

\begin{lemma}
Let $Q$ be a quiver and $\dd =(d^1, d^2 )$ a dimension filtration of $d$, then 
the fibres of the map $\pi_{\dd}$ are either empty or products of vector space Grassmannians. 
In particular they are all smooth and connected if they are non-empty.\\
Furthermore, we have $\Bild \pi_{\dd} = \{ M\in \rep_d (kQ/J^2) \mid \Dim c(M) \leq \dd \leq \Dim r(M) \}$. 
\end{lemma}

\begin{proof}
Let $M\in \rep_d (kQ/J^2)$. By Lemma \ref{fibre lemma} we have 
\[ \pi_{\dd}^{-1}(M) \cong \Gr_{\NtwoQ/(e)} \binom{r(M)/c(M)}{\dd- \Dim c(M)} .\]
Since $\NtwoQ/(e)$ is semi-simple, the quiver Grassmannian on the right is either empty or a product of vector space Grassmannians. A point $M \in \rep_d(kQ/J^2)$ is contained in the image if and only if the fibre over this point is non-empty, and this is precisely the case when the pointwise dimension inequality  $\Dim c(M) \leq \dd \leq \Dim r(M)$ is fulfilled. 
\end{proof}

\begin{rem}
Let $B $ be a finite dimensional $k$-algebra, $e\in B $ an idempotent element. Let $A=e B e$ and $M$ an $A$-module. Recall from \cite{FP} that we have an exact sequence of natural transformations 
\[ 
0\to c\to r \to iqr \to 0.
\]
We have a natural $k$-algebra homomorphism 
$
\psi_M \colon \End_A(M)\to \End_{B /(e)} (qr (M)).
$
In general it is neither injective nor surjective.  
\end{rem}

Let us return to $B =\NtwoQ$ and $e=\sum_{i\in Q_0}e(i_2)$.
The isomorphism $\pi_{\dd}^{-1}(M)\to \Gr_{\NtwoQ/(e)} \binom{qr(M)}{\dd- \Dim c(M)} $ is $\Aut_{kQ/J^2}$-equivariant where the group operation on the $\NtwoQ/(e)$-quiver Grassmannian is induced by the homomorphism $\Aut_{kQ/J^2}(M) \to \Aut_{\NtwoQ /(e)}(qr(M)) $.

\begin{lemma}
For every $kQ/J^2$-module $M$, the map $f\mapsto qr(f)$ 
\[ 
\End_{kQ/J^2} (M) \to \End_{\NtwoQ/(e)}(qr(M))=\End_{\NtwoQ/(e) } (r(M)/c(M))
\] 
is surjective. 
\end{lemma}

\begin{proof}
Given a $kQ/J^2$-module $M$, the module $qr(M)$ is a the maximal summand of $\Top(r(M))$ which is supported on $\{i_1\mid i\in Q_0\}$. For an indecomposable $\NtwoQ$-module $Y$ one has that $\Top(Y)$ is supported in $\{i_1\mid i\in Q_0\}$ if and only if $Y=P(i_1)$ or $Y=S(i_1)$ for some $i\in Q_0$. 
Since $S(i_1)$ is not in $\Bild r$ we conclude that $qr(M)$ coincides with $q(Y^\prime )$, where $Y^\prime \subset r(M) $ is a maximal summand isomorphic to a direct sum of projectives of the form $P(i_1)$ for $i \in Q_0$. 
The projection onto a summand gives a surjective map  
\[ 
\End_{kQ/J^2}(M)= \End_{\NtwoQ}(r(M))\to \End_{\NtwoQ}(Y^\prime ).  
\]
Hence it is enough to prove that $\End_{\NtwoQ}(Y^\prime ) \to \End_{\NtwoQ/(e)} (q(Y^\prime))$ is surjective. 
We have $\Hom_{\NtwoQ} (P(i_1), P(j_1))=0$ for $i\neq j$, thus it is enough to prove this for $Y^\prime =P(i_1)^n$. 
But by definition of $\NtwoQ$ we have $\End_{\NtwoQ} (P(i_1)) = M_n(k) = \End_{\NtwoQ/(e)} (S(i_1))$, therefore the map is surjective.   
\end{proof}

Since $\Aut_{\NtwoQ/(e)}(r(M)/c(M))$ is a product of general linear groups which operates transitively on $\Gr_{\NtwoQ/(e)} \binom{r(M)/c(M)}{\dd- \Dim c(M)}$, the lemma implies that the group $\Aut_{kQ/J^2}(M)$ operates transitively on the quiver Grassmannian. 

We summarise the previous discussion in the following proposition. 
\begin{prp} \label{s2RO}
Let $Q$ be a quiver and $\dd =(d^1, d^2=d )$ a dimension filtration, then ($Q,\dd)$ has a Richardson orbit if and only if there is a dense $\GL_d$-orbit in $\Bild \pi_{\dd}$. Furthermore, for every point $M\in \rep_d(kQ/J^2)$, the fibre 
$\pi_{\dd}^{-1}(M)$ is non-empty if and only if $\Dim c(M) \leq \dd \leq \Dim r(M)$ is fulfilled. If it is non-empty, $\pi_{\dd}^{-1}(M)$ is a product of Grassmannians and the group operation of $\Aut_{kQ/J^2}(M)$ is transitive. 
\end{prp}


\subsubsection{When is $kQ/J^2$ representation-finite?} 
 
One way to guarantee that $R_d(kQ/J^2)$ has a dense orbit is if $kQ/J^2$ is of finite representation type.
There is a method to check this in \cite{Ga2}, namely the separation quiver.
Given a quiver $Q = (Q_0,Q_1)$, the separation quiver of $Q$ has as vertices the disjoint union $Q_0 \cup Q_0^*$, where $Q_0^* = \{i^* \mid i \in Q_0 \}$ is a copy of $Q_0$, 
the arrows are $\{ (a' \colon i \to j^*) \mid (a \colon i \to j) \in Q_1 \}$.
There is a bijection between indecomposable isomorphism classes of non-simple $kQ/J^2$ modules and isomorphism classes of non-simple representations of the separation quiver of $Q$.
That implies $kQ/J^2$ is of finite representation type if and only if the separation quiver is of finite representation type.
By Gabriel's theorem that holds true if and only if all its components are of Dynkin type.
\begin{example}
Consider a quiver $Q$, drawn on the left, and its separation quiver beside it.
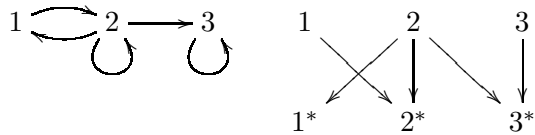
\[
\xymatrix{
1 \ar@/^/[r] & \ar@/^/[l] 2 \ar@(dl,dr) \ar[r] &  3 \ar@(dl,dr)
}
\qquad
\xymatrix{
1 \ar[dr] & \ar[dl] 2 \ar[dr] \ar[d] & 3 \ar[d] \\
1^* & 2^* & 3^*
}
\]
The separation quiver is of type $E_6$, so it is of finite representation type. Hence $kQ/J^2$ is of finite representation type.
\end{example}


\section{Examples}
\label{sec:examples}


Let $(Q, \dd)$ be a quiver and a dimension filtration such that there is a dense $\GL_d$-orbit in the image of $\pi_{\dd}$. 
Let $M$ be a point in this orbit and assume that the $kQ/J^s$-module $M$ is indecomposable, rigid and fulfils $\End_{kQ/J^s} (Q)=k$. 
Then the operation of $\Aut_{kQ/J^s} (M)=k^*$ on the fibre $\pi^{-1}_{\dd}(M)$ is trivial. 
This means that if there is a Richardson orbit for $(Q, \dd)$, then $\pi_{\dd} $ is a desingularisation of $\Bild \pi_{\dd}$. \\
However this does not have to be the case, as we see in the following example.

\subsection{Embedding $\widetilde{D_4}$}
\label{subsec:D4}

Let $Q$ be the quiver of type $D_4$ with a sink in the middle
\[
\raisebox{-12pt}{Q =} \;
\xymatrix@R=0.5em{
1 \ar[dr] & & 2 \ar[dl] \\
 &3 & \\
  & & 4. \ar[ul]
}
\] 
We consider $s=3$ and the dimension filtration 
\[ 
\dd = \left(\begin{smallmatrix} 0 && 0 \\ 
   & 1& \\ && 0 \end{smallmatrix}, \begin{smallmatrix} 0 && 0 \\ 
   & 2& \\ && 0 \end{smallmatrix},\begin{smallmatrix} 1 && 1 \\ 
   & 2& \\ && 1\end{smallmatrix}  \right) .
\]
Since $kQ/J^3= kQ$ is representation-finite, there is a dense orbit in $\Bild \dd $. 
Consider the following subquiver $Q'$ of $Q^{(3)}$:
\[ 
\xymatrix@R=0.7em{
1_3 \ar[dr] & & 2_3 \ar[dl] \\
  & 3_2 & \\
 3_1 \ar[ur]_{b(3_1)} & & 4_3. \ar[ul]
}
\] 
For any representation $M$ of $Q'$ we get an $N_3(Q)$-module $N$ by adding the identity for the arrow $b(3_2)$, and the zero vector space at the vertices $i_t$ for $i= 1,2,4, \; t=1,2$.
 
This gives a faithful embedding of $kQ' \bil \fdmod$ in $N_2(Q) \bil \fdmod$, and $N$ is $\Delta$-filtered if the map $M_{b(3_1)}$ is injective.
But there there is no Richardson orbit for $(Q,\dd )$, because the quiver 
representations of type $\widetilde{D_4}$ of dimension $ d' = \left( \begin{smallmatrix} 1 && 1 \\ &2& \\ 1&&1 \end{smallmatrix} \right)$ have no dense orbit, and $M_{b(3_1)}$ is injective on an open subset of $\rep_{d'}(Q')$. 
Since $Q$ is representation finite, $\Bild \pi_{\dd }$ is an orbit closure, so the condition that fails is actually the relaxed Richardson property.


\subsection{The Kronecker quiver}
\label{subsec:Kronecker}
Let $Q$ be the $2$-Kronecker quiver $\xymatrix{1 \ar@<1ex>[r] \ar@<-1ex>[r] &2}$. 
If we take $\dd =((0,1), (1,1))$, then we claim $\Bild \pi_{\dd} = \rep_{(1,1)}$; for  every $kQ$-module 
$M= (\xymatrix{k \ar@<1ex>[r]^{\lambda } \ar@<-1ex>[r]_{\mu} & k} )$ 
of dimension vector $(1,1)$ there is a $\dd$-dimensional stable $\NtwoQ$-module
\[ 
\raisebox{-18pt}{N =} \; \xymatrix{ k \ar@<1ex>[dr]_{\mu} \ar[dr]^{\lambda} & k \\ 0 \ar[u] & k \ar[u]_{1}   } 
\]
with $eN=M$. 
It is well-known that there is no dense $\GL_{(1,1)}$-orbit in $\rep_{(1,1)}$, thus there exists no Richardson orbit for $(Q, \dd)$. 
If we $\lambda $ and $\mu$ are not both zero, then $\dd = \Dim c(M)$. 
Thus $(Q,\dd)$ has the relaxed Richardson property by Theorem \ref{thm:relax}.


\subsection{Auslander algebras}
\label{subsec:Auslander}
As has been emphasised before, $\NsQ$ is intended to generalise certain aspects of the Auslander algebra of $k(x)/\langle x^{s} \rangle$.
With the help of our quasi-hereditary structure we can tell exactly in what cases $\NsQ$ is actually an Auslander algebra.
\begin{prp}
Let $s \geq 2$, then $\NsQ$ is an Auslander algebra if and only if all connected components of $Q$ are an oriented cycle.
In this case the vertices in the $s$-th layer correspond to the projective-injective modules over $kQ/J^s$.
\end{prp}

\begin{proof}
Let us assume $\NsQ$ is the Auslander algebra of a representation finite algebra $\Lambda$ for $s \geq 2$.
From \cite[Prop. 2.7]{Ei} we see that if an Auslander algebra is simultaneously left and right strongly quasi-hereditary, then $\cF(\Delta)$ is precisely the full subcategory of torsionless modules. 
Since $\NsQ$ has such a quasi-hereditary structure \cite[Theorem 3]{Ei} implies that $\NsQ$ must be the Auslander algebra of a uniserial algebra $\Lambda$. 

For each vertex $i$ of $Q$, we have the injective envelope $P(i_1) \to I(i_s)$.
Since $\NsQ$ has dominant dimension at least 2 this implies $I(i_s)$ is in fact projective. 
It is well known that for Auslander algebras that there is an equivalence from $\Lambda \bil \modu$ to the injective $\NsQ$-modules, with projective-injective $\NsQ$ modules corresponding to projective $\Lambda$-modules, thus the filtration
\[
I(i_s) \twoheadrightarrow I(i_{s-1}) \twoheadrightarrow \cdots \twoheadrightarrow I(i_1)
\]
corresponds to a filtration of a projective $\Lambda$-module.
Recall that the indecomposable $\Lambda$-modules correspond to simple $\NsQ$-modules in  a canonical way.
Since $\Lambda$ is uniserial all modules in this sequence are indecomposable and distinct, so for each $i\in Q$ we get $s$ pairwise non isomorphic $\Lambda$-modules generated by the same projective module.
Since the idempotents of $\NsQ$ must correspond bijectively to indecomposable $\Lambda$-modules, all indecomposable modules arise in this way. 
In particular all projective $\Lambda$-modules have Loewy-length $s$.
But by the classification of uniserial algebras that shows $\Lambda$ is of the form $kQ/J^s$ where $Q$ is an oriented cycle.
It was already noted that the vertices in the $s$-th layer correspond to the projective-injective $kQ/J^s$-modules.

The converse is clear, note that the vertex $i_t$ of $\Qs$ corresponds to the quotient of $P(i)$ of length $t$. 
\end{proof}

Since $kQ/J^s$ is a representation-finite self-injective algebra, every intermediate extension $c(M)$ for $M \in \rep_d(kQ/J^s)$ is rigid (cf. \cite{CBSa}). 
In fact $c$ gives an equivalence of categories $kQ/J^s \bil \modu \to \add(T)$, where $T$ is the characteristic tilting module of $\NsQ$.
Consequently, for every $\dd = \Dim c(M)$, the map $\pi_{\dd}$ is a desingularisation of $\overline{\cO_M}$, which is already studied in \cite[Thm 7.5]{CBSa}. 

On the other hand the functor $r$ is simply the standard projective embedding of $kQ/J^s \bil \modu$ into the module category of its Auslander algebra.
This follows from the fact that $e \Gamma$ is the an Auslander generator for $kQ/J^s \bil \modu$, and $r = \Hom_{kQ/J^s} (e\Gamma,-)$.
Consequently $r(M)$ is rigid for all $M \in kQ/J^s \bil \modu$.


\subsection{Lifting projective modules}
\label{subsec:projectives} 
This provides examples of quiver-graded Richardson orbits where the generic fibre of $\pi_{\dd}$ is not zero-dimensional. 
Let $Q$ be any quiver and $s\geq 1$ arbitrary. Then any projective $\NsQ$-module is rigid, and from general facts about quasi-hereditary algebras it is in the category $\cN$. 
So if $\dd $ is a dimension vector of a projective $\NsQ$-module, then $(Q, \dd )$ has a Richardson orbit. \\
If $P = \NsQ e(i_t)$ for some $t\in \{1, \ldots , s\}$, then we have 
$e P \cong (kQ/J^t)e(i)$ as $kQ/J^s$-modules. 
For $s$ large enough this gives many rigid modules $P$ with 
\[ ce(P )\neq P, \quad P  \neq re(P) .\]
On the other hand, we can calculate for $\dd =\Dim P, d= \Dim eP$ 
\[ 
\begin{aligned}
\dim \pi_{\dd}^{-1}(eP) &=\dim \RF-\dim \rep_d \\
& = \dim_k \Hom_{kQ/J^s}( eP, eP ) - \dim_k \Hom_{\NsQ}(P, P).
\end{aligned}
\]
For $s\geq 2$ it is easy to find examples where this is $>0$.


\subsection{Algorithm for $Q = \A_2$} 
\label{subsec:algor}

For the quiver $Q=\A_2$ and arbitrary but fixed $s\geq 1$ the category $\mathcal{N}$ has only finitely many indecomposable objects. 
This implies that for all dimension filtrations $\dd $ there exists a Richardson orbit for $(Q,\dd )$, we give an algorithm calculating it. 
We write $\A_2 = \xymatrix{x \ar[r] & y}$ and let $\dd = (d_1,\dots, d_s)$ be a dimension filtration, i.e. a dimension vector for $N_s(\A_2)$ with $d_t \leq d_{t+1}$ pointwise for $t=1,\dots,s-1$. 
Up until now $i,j$ have denoted vertices of $Q$, but due to a a high demand of different indices we repurpose the indices $i,j$ here to denote layers in $\Qs$.
For $N_s(\A_2)$-modules $M,N$, we write $[M] := \underline{\dim} M$ as well as
$\left[M,N \right]:= \dim \Hom_{N_s(\mathbb{A}_2)}(M,N)$ and 
$\left[M,N \right]^1:= \dim \ext^1_{N_s(\mathbb{A}_2)}(M,N)$. 
There is unique vector $\delta_{\dd} = ((\hat{x}_i)_{i=1}^s,(\hat{y}_j)_{j=1}^s) \in \N_0^{2s}$ such that
\[
	\dd = \sum_{i=1}^s \hat{x}_i [\Delta(x_i)] + \sum_{j=1}^{s} \hat{y}_j [\Delta(y_j)].
\]
We call this the $\Delta$-dimension vector corresponding to $\dd$.
We write $\delta_{x_i} := \delta_{[\Delta(x_i)]}$ and $\delta_{y_j} := \delta_{[\Delta(y_j)]}$. 
For $i>j$ we denote by $E(i,j)$ the indecomposable module with $\Delta$-dimension vector $\delta_{x_i} + \delta_{y_{j}}$.

\begin{prp}
The following algorithm returns a rigid $N_s(\A_2)$-module.

Let $\dd$ be a dimension filtration and let $\delta_{\dd} = ((\hat{x}_i)_{i=1}^s,(\hat{y}_j)_{j=1}^s)$ denote the corresponding $\Delta$-dimension vector. Let $M=0$ be the trivial $N_s(\A_2)$-module.
We execute the following steps:
\begin{enumerate}[$(1)$]
\item
If $\hat{x}_i=0$ for all $i=1,\dots,s$ go to step $(3)$.
Otherwise let $i$ be minimal such that $\hat{x}_i \neq 0$ and go to step $(2)$.

\item
If $\hat{y}_{j} = 0$ for all $j < i$, replace $M$ with $M \oplus \Delta(x_i)$ and $\delta_{\dd}$ with $\delta_{\dd} - \delta_{x_i}$, then go back to step $(1)$.
Otherwise let $j$ be maximal such that $j < i$ and $\hat{y}_{j} \neq 0$.
Replace $M$ with $M \oplus E(i,j)$ and $\delta_{\dd}$ with $\delta_{\dd} - \delta_{x_i} - \delta_{y_{j}}$.
Then go back to step $(1)$.

\item
Return the module $M \oplus \bigoplus_{j=1}^s \Delta(y_j)^{\hat{y}_j}$.
\end{enumerate}
\end{prp}

\begin{proof}
Denote the module returned by the algorithm by $M$. 
We write $M=M_1\oplus M_2 \oplus M_3$, 
with $M_1\in \add \bigoplus_{i=1}^s \Delta(x_i)$, 
$M_2 \in \add\bigoplus_{i>j} E(i,j)$ and $M_3 \in \add \bigoplus_{j=1}^s \Delta(y_j)$.  \\
We observe that $T= \bigoplus_{i=1}^s \Delta (x_i)\oplus \bigoplus_{i=1}^{s-1} E(i+1,i) \oplus \Delta (y_s)$ is the characteristic tilting module of $N_s(\A_2) \bil \fdmod$, and that $\Delta (y_j)$ is projective for $j=1,\dots,s$. 
We use these properties and apply appropriate Hom-functors to short exact sequences of the form
\[ 
0\to \Delta (y_j ) \to E(i,j) \to \Delta (x_i) \to 0,
\] 
to calculate all missing dimensions of Ext-groups.

We know $\ext^1(M,M_1) = 0$, because $M$ is in $\cF(\Delta)$ and $M_1 \in \add(T)$. 
We also have $\ext^1(M_3,M) = 0$, because $M_3$ is projective.
It remains to show:
\[ 
\begin{aligned}
& \textbf{(a)}\; \ext^1 (M_1, M_3)=0, \quad &\textbf{(c)}\; \ext^1(M_2, M_3)=0, \\
& \textbf{(b)}\; \ext^1 (M_1, M_2)=0, \quad &\textbf{(d)} \;\ext^1(M_2, M_2)=0 .
\end{aligned}
\]

\paragraph{Case (a):}
We have $\left[ \Delta (x_i), \Delta (y_j)\right]^1=1$ if and only if $i>j$. If $\Delta(x_i)\in \add (M_1)$ and $i > j$ we see that during the iteration of step (2) that yields $\Delta(x_i)$, we have $\hat{y}_j = 0$. 
But then $\Delta (y_j) \notin \add (M_3)$, because $\hat{y}_j$ is still zero in step (3).

\paragraph{Case (b):}
From the long exact sequence obtained by applying $\Hom_{\cN}(\Delta(x_t),-)$ to the short exact sequence above, we have $\left[\Delta (x_t), E(i,j)\right]^1=1$ if and only if $i>t>j$. 
Let $\Delta (x_t) \in \add (M_1)$ and let $i>t>j$. 
By the condition in step (2) we have $\hat{y}_j =0$ during the step that yields $\Delta(x_t)$. 
If $E(i,j) \in \add(M_2)$, then that is yielded in a later iteration of step (2), because $i>t$. But that can't be since $\hat{y}_j$ is still zero in that step.

\paragraph{Case (c):}
From the long exact sequence obtained by applying $\Hom_{\cN}(-,\Delta(y_t))$ to the short exact sequence above we have $\left[E(i,j), \Delta (y_t)\right]^1=1$ if and only if $i>t>j$. 
Let $E(i,j)\in \add (M_2)$ and let $i>t>j$.
In the step that yields $E(i,j)$ we must have $\hat{y}_t =0$, otherwise $j$ is not maximal. But then $\hat{y}_t=0$ in step (3) also, thus $\Delta(y_t) \notin \add M$.

\paragraph{Case (d):}
We apply $\Hom_{\cN}(-, E(t,l))$ to the short exact sequence above to obtain
\[ \left[ E(i,j), E(t,l)\right]^1= 
\left[ E(i,j), E(t,l)\right] - \left[\Delta (y_j), E(t,l)\right] + \left[\Delta(x_i), E(t,l)\right]^1.
\]
We first note $\left[ E(i,j), E(t,l)\right] =1$ if and only if one has $i\geq t, j\geq l$. 
Also we have $\left[ \Delta (y_j), E(t,l)\right] =1$ if and only if $j\geq l$. From case (b) we know $\left[\Delta(x_i), E(t,l)\right]^1=1$ if and only if $t>i>l$. We conclude 
\[
\left[ E(i,j), E(t,l)\right]^1= 
\begin{cases} 1 ,& \text{if }t>i>l>j; \\
0,&\text{else. }
\end{cases}
\]
Let $E(i,j), E(t,l) \in \add (M_2)$ and assume $t>i>l>j$.
From step $(1)$ in the algorithm we see $E(i,j)$ is obtained before $E(t,l)$.
Thus $\hat{y}_j,\hat{y}_l > 0$ at the start of the iteration of step $(2)$ that yields $E(i,j)$.
But since $i>l>j$ that is a contradiction to $j$ being maximal such that $\hat{y}_j > 0$ and $j > i$. Thus $E(i,j)$ and $E(t,l)$ cannot both be summands of $M$.
\end{proof}

\providecommand{\bysame}{\leavevmode\hbox to3em{\hrulefill}\thinspace}
\providecommand{\MR}{\relax\ifhmode\unskip\space\fi MR }
\providecommand{\MRhref}[2]{%
  \href{http://www.ams.org/mathscinet-getitem?mr=#1}{#2}
}
\providecommand{\href}[2]{#2}

\end{document}